\documentclass[intlimits]{amsart}

\usepackage[all,hyperref,numberbysection]{bi-discrete}

\usepackage{tikz}
\usetikzlibrary{positioning,arrows.meta,calc,decorations.markings}
\providecommand{\calA}{\mathcal A}
\providecommand{\calN}{\mathcal N}

\usepackage[backend=biber,maxbibnames=5,maxalphanames=5,style=alphabetic,bibencoding=utf8,giveninits,url=false,isbn=false]{biblatex}
\AtBeginBibliography{\small}
\usepackage{tikz}
\usepackage{esint}
\renewcommand{\dashint}{\fint}
\newcommand{\px}{{p(\cdot)}}
\newcommand{\pdx}{{p'(\cdot)}}
\newtheorem*{question*}{Question}
\newcommand{\dominatedby}{\ensuremath{\preceq}}
\allowdisplaybreaks

\begin{document}

\title{Variable Exponent Regularity via Muckenhoupt Condition}

\author{Daviti Adamadze}
\address{Fakult\"{a}t f\"{u}r Mathematik, Universit\"{a}t Bielefeld, Bielefeld 33615, Germany}
\email{dadamadz@math.uni-bielefeld.de}

\author{Anna Kh. Balci}
\address{Fakult\"{a}t f\"{u}r Mathematik, Universit\"{a}t Bielefeld, Bielefeld 33615, Germany}
\email{anngremlin@gmail.com}

\author{Lars Diening}
\address{Fakult\"{a}t f\"{u}r Mathematik, Universit\"{a}t Bielefeld, Bielefeld 33615, Germany}
\email{lars.diening@uni-bielefeld.de}


\thanks{Daviti Adamadze gratefully acknowledges financial support from the Deutsche Forschungsgemeinschaft (DFG, German Research Foundation) through IRTG 2235 (Project No. 282638148). Anna Kh. Balci and Lars Diening gratefully acknowledge financial support from the Deutsche Forschungsgemeinschaft (DFG, German Research Foundation) through SFB 1283/2 2021 (Project No. 317210226).
}

\begin{abstract}
  For the first time, we establish higher integrability of the gradient and local $L^\infty$ estimates of weak solutions to the $p(x)$-Laplacian without assuming $\log$-H\"older continuity of~$p$.  Instead, we establish these results for exponents satisfying a generalized Muckenhoupt condition. This admits discontinuous exponents acting as pointwise multipliers of the BMO class. Our framework bridges the gap between classical $p(x)$-regularity and weighted Muckenhoupt theory, providing a new foundation for the analysis of differential equations with highly irregular non-standard growth.
\end{abstract}

\maketitle

\tableofcontents

\section{Introduction}
\label{sec:introduction}

In this paper we consider two problems of the $p(x)$-Laplacian, one with non-zero right-hand side
\begin{alignat}{2}
  \label{eq:px-laplacian_non_zero}
	-\divergence A(\cdot,\nabla u) &= -\divergence (\abs{G}^{{\px-2}}G) &\qquad &\text{in} \ \, \Omega,
  \\
  \intertext{and the other one with zero right-hand side}
  \label{eq:px-laplacian_zero}
  -\divergence A(\cdot,\nabla u) &= 0 &\qquad &\text{in} \ \, \Omega,
\end{alignat}
where $A(x,\nabla u) = \abs{\nabla u}^{p(x)-2} \nabla u$ and $\Omega \subset \RRn$ is some domain. Our goal is to establish higher integrability of~$\nabla u$ for solutions of~\eqref{eq:px-laplacian_non_zero} and local $L^\infty$-estimates for~$u$ for solutions of~\eqref{eq:px-laplacian_zero} under very mild conditions on the variable exponent~$p$, weaker than those known before. In particular, we will use a Muckenhoupt-type condition rather than the well-known $\log$-H\"older continuity condition.
Note that solutions to~\eqref{eq:px-laplacian_zero} are (local) minimizers of the energy
\begin{align}\label{eq:variable_energy}
  \mathcal{F}(v) = \int\nolimits_{\Omega} \frac{1}{p(x)} |\nabla v|^{p(x)} \, dx.
\end{align}

In the study of variational problems and differential equations with non-standard $p(x)$-growth, the regularity theory has historically been closely tied to log-H\"older continuity of the exponent, see~\eqref{eq:log-hoelder}. 
Due to the delicate nature of non-standard growth, the log-H\"older continuity has generally been considered to be sharp. To highlight the severity of this restriction, Acerbi and Mingione \cite[p.~123]{AcMi1} emphasized that "in general, dropping it causes the loss of any type of regularity of minimizers, like H\"older continuity and even higher integrability." The goal of our article is to drop the widely accepted log-H\"older continuity condition and move to a weaker Muckenhoupt-type condition.

The study of problems with non-standard growth was pioneered by Marcellini, who established the first regularity results for so-called $(p,q)$-growth problems \cite{Mar}, and by Zhikov, who investigated variational problems with $p(x)$-growth and the Lavrentiev phenomenon \cite{Zhik5}. Shortly thereafter, the H\"older continuity of the solutions was independently established in \cite{Piatcoscia_Hoelder_Continuity,Alkhutov_Harnack_Holder_px,Fanzhao_De_Giorgi_classes}. Chiad\`o Piat and Coscia~\cite{Piatcoscia_Hoelder_Continuity} proved it under the assumption that $\px \in W^{1,s}(\Omega)$ for $s>n$, while Alkhutov~\cite{Alkhutov_Harnack_Holder_px} and Fan and Zhao~\cite{Fanzhao_De_Giorgi_classes} utilized the weaker log-H\"older continuity condition.

Employing the same condition, Fan and Zhao developed De Giorgi-type methods to study quasiminimizers in the variable exponent setting \cite{Fanzhao_De_Giorgi_classes,Fanzhao_discont_exponent}. Furthermore, Acerbi and Mingione obtained higher-order regularity results for minimizers with $p(x)$-growth \cite{AcMi1}. Subsequently, these types of results were extended to a slightly more general context in~\cite{Adamowicz_Toivanen_Hoelder_nonstandard}. Calder\'on--Zygmund type regularity results were obtained for solutions to equation \eqref{eq:px-laplacian_non_zero} locally in \cite{AcerbiMingione2005} and globally in \cite{DieSch14}. Regarding the Lavrentiev phenomenon, a systematic approach was developed for the variable exponent and even more general models in \cite{Zhik5,ELM,BDS,BalciDieningSurnachev2025}.

In this paper we develop new techniques that allow us to push the results beyond the scope of log-H\"older continuous exponents. In particular, we will work in a framework similar to that for Muckenhoupt weights. To understand the meaning of this, let us spend a few words on the weighted $p$-Laplacian (and the weighted Laplacian for $p=2$) given by
\begin{align}
  \label{eq:weighted-p-laplacian}
  -\divergence \big( \mu(x) \abs{\nabla u}^{p-2} \nabla u\big) &= 0
\end{align}
as well as the double phase model
\begin{align}
  \label{eq:double-phase}
  -\divergence \big( \abs{\nabla u}^{p-2} \nabla u + \mu(x) \abs{\nabla u}^{q-2} \nabla u\big) &= 0
\end{align}
with $1 < p < q < \infty$.  The corresponding energies are $\int_\Omega \phi(x,\abs{\nabla u(x)})\,dx$ with $\phi(x,t) = \frac 1p \mu(x) t^p$ and $\phi(x,t) = \frac 1p t^p + \frac 1q \mu(x) t^q$, respectively. In the case of the $p(x)$-Laplacian we have $\phi(x,t) = \frac{1}{p(x)} t^{p(x)}$.
Solutions to the weighted Laplacian (i.e., $p=2$) with $\mu \in A_2$ have been investigated in the famous paper~\cite{FabesKenigSerapioni1982}. The case $p \in (1,\infty)$ and $\mu \in A_p$ has been studied in~\cite{CruzUribeMoenNaibo2013}. In both papers the local H\"older regularity of solutions has been shown.

The generalization of the Muckenhoupt condition to the $p(x)$-Laplacian or the double phase model is not straightforward. Over the last two decades it turned out that the most promising and natural approach is to define the Muckenhoupt condition in terms of
\begin{align}
  \label{eq:general-muckenhoupt}
  [\phi]_{\mathcal{A}} \coloneqq \sup_Q \frac{\norm{\indicator_Q}_{\phi}
  \norm{\indicator_Q}_{\phi^*}}{\abs{Q}} < \infty,
\end{align}
where $\norm{\cdot}_\phi$ is the norm associated with the energy of the model, $\norm{\cdot}_{\phi^*}$ is the norm of the associated dual space, and the supremum is taken over all cubes, see~\eqref{eq:pinA}. This condition agrees with the classical Muckenhoupt condition, when applied to $\phi(x,t) = \frac 1p \mu(x) t^p$. Moreover, every $\log$-H\"older continuous exponent satisfies~\eqref{eq:general-muckenhoupt}. The class of variable exponents satisfying~\eqref{eq:general-muckenhoupt} is, however, much larger, see~\eqref{eq:example-lerner}.

The Muckenhoupt condition~\eqref{eq:general-muckenhoupt} has been used in~\cite{adamadzedieningkopalianiok} to show H\"older continuity of the solutions to the double phase model. The goal of our paper is to extend the results (partially) to the $p(x)$-Laplacian. This is part of the long-term goal to provide a unifying theory that includes the weighted $p$-Laplacian, the double phase model, the $p(x)$-Laplacian as well as more general models. 
In Figure \ref{fig:boundedness-conditions} we summarize the development so far.

\begin{figure}[ht]
	\centering
	\begin{tikzpicture}[
		scale=0.75, transform shape,
		every node/.style={align=center},
		box/.style={
			draw, rounded corners, inner sep=5pt,
			minimum width=3.4cm, font=\small
		},
		arrow/.style={-{Latex[length=2mm]}, thick},
		equiv/.style={<->, >=Latex, thick},
		double slash/.style={
			-{Latex[length=2mm]}, thick,
			postaction={decorate},
			decoration={
				markings,
				mark=at position 0.5 with {
					\draw[-, thick, solid, line cap=round] (-2pt, -3pt) -- (0pt, 3pt);
					\draw[-, thick, solid, line cap=round] (1pt, -3pt) -- (3pt, 3pt);
				}
			}
		}
		]

		\node[box] (M) at (0, 0) {\textbf{Boundedness of $M$}};
		
		\node[box] (DP) at (0, 2.5) {\textbf{Double phase}\\$\varphi \in \calA$\\ \cite{adamadzedieningkopalianiok}};

		\node[box] (px) at (-4.5, 2.5) {\textbf{Variable exponents}\\$p \in \calA \cap \calN$\\\cite{Lerner2010questions,CUF,AdamadzeDieningKopaliani2026}};
				
		\node[box] (G) at (4.5, 2.5) {\textbf{General Model}\\$\phi \in \calA + \text{???}$
      \\
    open problem};
		\draw[arrow] (DP.south) -- ([xshift=0cm]M.north);

		\draw[arrow] (px.south) -- ([xshift=-1.5cm]M.north);
		\draw[arrow] (G.south) --  ([xshift=1.5cm]M.north);
		
	\end{tikzpicture}
	\caption{Towards a general condition for the boundedness of $M$.}
	\label{fig:boundedness-conditions}
\end{figure}
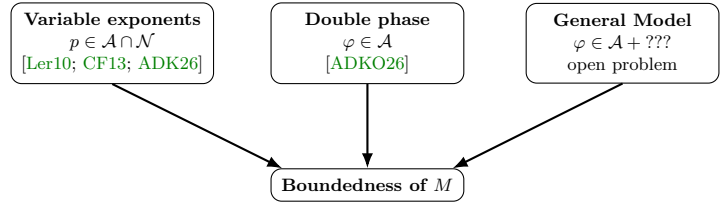

Let us mention here that the analysis of the double phase model subject to the condition~\eqref{eq:general-muckenhoupt} in~\cite{adamadzedieningkopalianiok} is quite difficult, but still much simpler than the one for variable exponents. This can be seen by comparing the generalized Jensen's inequalities of the two models. In particular, for the double phase model, a generalized Jensen's inequality similar to that in Theorem~\ref{thm:generalized-jensen} holds (see~\cite[Theorem 2.5]{adamadzedieningkopalianiok}), but without the need for the extra~$\abs{Q}^\epsilon$ term.

A different approach to a unifying theory was developed by Harjulehto, H\"ast\"o and Ok. In~\cite{HastoHarjulehto2019}, Harjulehto and H\"ast\"o extend the theory on variable exponent spaces of~\cite{DHHR} to a larger class of generalized Orlicz spaces. In~\cite{HastoOk2022} they use this technique to prove regularity results for minimizers of these more general models. The conditions that they introduce are immediate generalizations of the $\log$-H\"older continuity of~$p(\cdot)$ to more general models. This is a very interesting approach, but it excludes the case of the weighted Laplacian and weighted $p$-Laplacian with Muckenhoupt weights. It is the purpose of our paper to merge this approach with the world of Muckenhoupt weights. Our paper~\cite{adamadzedieningkopalianiok} is a first step in this direction restricted to the double phase model.

Our technique used to generalize the Muckenhoupt class to other models has its origins in the work~\cite{Diening2005}, slightly polished in~\cite{DHHR}. In this work, a general theory for the Musielak--Orlicz setting was initiated, analogous to the Muckenhoupt theory for weighted Lebesgue spaces. Of particular interest is the left-openness property of the class $\mathcal{A}_{\text{global}}$ (defined precisely below), which generalizes the corresponding left-openness of the classical Muckenhoupt class $A_p$. This deep result is essential for the fine analysis of differential equations modeled on the $p(x)$-Laplacian. Note that~\cite{Diening2005} provides a characterization of the boundedness of the maximal operator in the context of variable exponent spaces. However, the general setting of Musielak--Orlicz spaces is not yet complete.

Let us also compare our results to other techniques that allow one to treat the $p(x)$-Laplacian without the $\log$-H\"older condition or a Muckenhoupt condition. One approach is to work within the context of $p$-$q$ growth problems intensively investigated by Marcellini, e.g.~\cite{Mar}. The paper~\cite{hirsch_schaeffner_Growth_conditions_and_regularity} provides the currently strongest result in this direction. Translated to our setting it says that if
\begin{align}
  \label{eq:hirsch-schaeffner}
  \frac{1}{\sup p} &\geq \frac{1}{\inf p} - \frac{1}{n-1},
\end{align}
then solutions are locally bounded. This is a nice result, but it excludes some good exponents. In particular, any $\log$-H\"older continuous exponent provides local boundedness of the solution without the need of~\eqref{eq:hirsch-schaeffner}, see~\cite{Alkhutov_Harnack_Holder_px,Fanzhao_De_Giorgi_classes}. Moreover, the $p$-$q$-growth approach has its own difficulties, when aiming for higher integrability of the gradients or H\"older continuity of solutions.

Another possibility to allow for more general exponents is to use continuous~$p$ and to apply the $p$-$q$-growth theory only locally, see for example~\cite{Fanzhao_discont_exponent}.
However, this approach only allows us to obtain
 $L^\infty$-estimates on very small balls whose diameter is coupled to the modulus of continuity of the exponent~$p$. When passing back to medium-sized balls, it is not anymore possible to control the constants in a good way. We believe that the
  approach using a Muckenhoupt condition as in~\eqref{eq:general-muckenhoupt} is the most natural one to get a solid theory on the local boundedness of solutions. In
   particular, we expect that the results of our paper will open the road to local H\"older continuity. Moreover, our technique immediately allows us to get higher integrability of the gradients.

Overall, we believe that the technique in this paper using the Muckenhoupt condition~\eqref{eq:general-muckenhoupt} is a big step towards a unified theory for models based on generalized Orlicz spaces. The final goal in this context is to prove H\"older continuity of the solutions as well as higher integrability of the gradients.

The paper is organized as follows. In Section~\ref{sec:variable-exponents} we introduce the Muckenhoupt class for variable exponents also providing the historical context. Based on this condition we develop an (improved) generalized Jensen's inequality, which is the natural but non-trivial extension of Jensen's inequality to the setting of variable exponents, see Theorem~\ref{thm:generalized-jensen}. In Section~\ref{sec:Sobolev-Poincare} we derive certain Sobolev--\Poincare{} estimates based on the improved generalized Jensen's inequality. Then in Section~\ref{sec:Regularities} we apply these new tools to show higher integrability of the gradient and local boundedness of the solution.

%% ------------------------------------------------------------

\section{On Variable Exponent Spaces}
\label{sec:variable-exponents}

Before we start the regularity results we need to recall well-known properties of variable exponent spaces and also develop a few new ones. We mostly use the notation of~\cite{DHHR}. See~\cite{CUF} for another standard reference.

Although our subsequent analysis will primarily focus on domains~$\Omega \subset \RRn$, it is advantageous at this stage to work on the entire space~$\RRn$. Consequently, we assume that the variable exponent~$p$ is defined on all of~$\RRn$. We make this global assumption because general extension theorems for the class~$\mathcal{A}$ (which will be formally defined later) from a domain to the whole space are currently unavailable.

In the following we assume that $p \,:\, \RRn \to [1,\infty]$ with $p^-, p^+ \in [1,\infty]$ such that 
\begin{align*}
	p^- \le p(x) \le p^+ \qquad \text{for all $x \in \RRn$}.
\end{align*}
Note that throughout the PDE part of this paper we will always assume that
\begin{align}
	\label{eq:p-uniform-convex}
	1 < p^- \leq p^+ < \infty.
\end{align}

The variable exponent space $L^\px(\Omega)$ is defined as follows:
\begin{align*}
  L^{p(\cdot)}(\Omega) &\coloneqq \bigset{ f \in L^1_{\loc}(\Omega)\,:\, \norm{f}_\px < \infty},
\end{align*}
where the norm $\norm{\cdot}_\px$ is defined by
\begin{align*}
  \norm{f}_{\px} &\coloneqq \inf \biggset{ \lambda>0 : \int\nolimits_\Omega \left| \frac{f(x)}{\lambda} \right|^{p(x)} dx  \leq 1}
\end{align*}
with the convention that $t^\infty := 0$ for $t \in [0, 1]$ and $t^\infty := \infty$ for $t > 1$.

For $1<p^- \leq p^+ <\infty$ the space $L^\px(\Omega)$ is a uniformly convex Banach space~\cite[Theorem~3.4.9]{DHHR}. The generalized Orlicz-Sobolev space is $$W^{1,\px}(\Omega) := \{u \in L^\px(\Omega) : D_ju \in L^\px(\Omega) \text{ for } j=1,\dots,n\},$$ equipped with the norm $\|u\|_{W^{1,\px}(\Omega)} := \|u\|_{L^\px(\Omega)} + \|Du\|_{L^\px(\Omega)}$. We define $W_0^{1,\px}(\Omega) := W^{1,1}_0(\Omega) \cap W^{1,\px}(\Omega)$ under the same norm. On the other hand, one can show that $W_0^{1,\px}(\Omega) \coloneq \{ u \in W^{1,\px}(\Omega) : \tilde{u} \in W^{1,\px}(\mathbb{R}^n) \}$ for domains with a sufficiently regular boundary (e.g. Lipschitz boundary), where $\tilde{u}$ is the zero-extension of $u$ to $\mathbb{R}^n \setminus \Omega$. Before proceeding, let us fix some standard notation. For a measurable set~$E \subset \RRn$, we denote its $n$-dimensional Lebesgue measure by~$\abs{E}$ and its indicator function by~$\indicator_E$. If~$E \subset \RRn$ is a set of finite positive measure, the integral average is defined as	$\dashint\nolimits_E f(x) \, dx \coloneqq \frac{1}{\abs{E}} \int_E f(x) \, dx.$
Furthermore, for a variable exponent~$p$, its conjugate exponent~$p'$ is defined pointwise by the relation $ \frac{1}{p(x)} + \frac{1}{p'(x)} = 1,$ with the standard convention that $1/\infty = 0$. When we write $a \lesssim b$, we mean that there is a constant $c$ depending at most on the dimension $n$ and the structural constants $p^-$, $p^+$, $[p]_{\mathcal{A}}$, and $[p]_{\mathcal{N}}$ such that $a \leq c b$. Furthermore, we write $a \eqsim b$ if $a \lesssim b$ and $b \lesssim a$. We have the following H\"older inequality, see \cite[Lemma 2.6.5]{DHHR} 
\begin{align}
	\label{eq:holder}
	\int\nolimits_{\Omega} \abs{f(x) g(x)} \, dx \le 2 \norm{f}_{L^{p(\cdot)}(\Omega)} \norm{g}_{L^{p'(\cdot)}(\Omega)}.
\end{align}
For $f \in L^1_{\loc}(\RRn)$, the maximal operator~$Mf \colon \RRn \to [0,\infty]$ is defined by
\begin{align}\label{eq:max}
  (Mf)(x) &\coloneqq \sup_{Q \ni x} \dashint\nolimits_Q \abs{f(y)}\,dy,
\end{align}
where the supremum is taken over all cubes (balls) $Q \subset \RRn$. It is also possible to take the supremum over all balls centered at~$x$, which gives an operator that is equivalent to \eqref{eq:max} up to a constant. 

Let us define the logarithmic modulus of continuity by
\begin{align*}
  \omega_{\log}(t) &\coloneqq \frac{1}{\log(e+ \frac 1{t})}.
\end{align*}
We say that $\frac 1p$ is $\log$-H\"older continuous, in short $\frac 1p \in C^{\log}(\RRn)$, if
\begin{align}
  \label{eq:log-hoelder}
  \Bigabs{\frac{1}{p(x)} - \frac{1}{p(y)}} &\leq c_{\log}(1/p)\, \omega_{\log}(\abs{x-y}) \qquad \text{for all $x,y \in \RRn$}
\end{align}
for some constant $c_{\log}(1/p)\geq 0$.  Moreover, we say that $\frac 1p$ satisfies the $\log$-decay condition if there exists $p_\infty \in [p^-,p^+]$ with
\begin{align}
  \label{eq:decay}
  \Bigabs{\frac{1}{p(x)} - \frac{1}{p_\infty}} &\leq c_{\log}(1/p)\, \frac{1}{\log(e+ \abs{x})} \qquad \text{for all $x \in \RRn$}.
\end{align}
The set of exponents~$p$ such that $\frac 1p$ satisfies the $\log$-H\"older continuity condition and the decay condition is denoted in~\cite{DHHR} by $\mathcal{P}^{\log}(\RRn)$. It is by now well known that if $p \in \mathcal{P}^{\log}(\RRn)$ and $1<p^- \leq p^+ \leq \infty$, then $M$ is bounded on $L^{\px}(\RRn)$, see \cite[Theorem~4.3.8]{DHHR}. The first result regarding the boundedness of the maximal function using the log-H\"older continuity of the exponent was obtained in \cite{Die2}. The decay condition can be replaced by a weaker alternative, introduced by Nekvinda~\cite{Nekvinda2004}, namely
\begin{align}
	\label{eq:nekvinda}
	1 \in L^{r(\cdot)}(\RRn),
\end{align}
where $\frac1{r(x)} \coloneqq \abs{\frac 1{p(x)} - \frac 1{p_\infty}}$ for some~$p_\infty \in[p^-,p^+]$. A variable exponent~$p$ satisfying~\eqref{eq:nekvinda} is said to satisfy the Nekvinda condition, which we denote by~$p \in \mathcal{N}$. We define $[p]_{\mathcal{N}} \coloneqq \norm{1}_{r(\cdot)}$.  
Nevertheless, the combined conditions~$\frac 1p \in C^{\log}(\RRn)$ and $p \in \mathcal{N}$ are not necessary for the boundedness of~$M$. It turns out that the boundedness of~$M$ is equivalent to a weaker condition of Muckenhoupt-type, called class~$\mathcal{A}$ in~\cite[Definition~4.4.6]{DHHR} but here we will call it $\mathcal{A}_{\text{global}}$. This condition basically states that the averaging operator over families of disjoint cubes has to be bounded. 
A lot of good properties can be derived from this condition, like left-openness and boundedness of~$M$ on $L^{s\px}(\RRn)$ for some~$s>1$ and on $L^\pdx(\RRn)$, see again \cite[Theorem~5.7.2]{DHHR}.
We denote the operator norm of the averaging operator by $[p]_{\mathcal{A}_{\text{global}}}$.
Note that the condition $\mathcal{A}_{\text{global}}$ is necessary and sufficient for the boundedness of~$M$ on~$L^\px(\RRn)$ if~$1< p^- \leq p^+ < \infty$.

A slightly modified version of Lerner's example from~\cite{Lerner05} (see also~\cite[Example~5.1.8]{DHHR}) is given by
\begin{equation}
	\label{eq:example-lerner}
	p(x) = 2 - \epsilon\,\big(1+\sin(\log\log(e+\abs{x}+1/\abs{x}))\big)
\end{equation}
for some small~$\epsilon>0$. This variable exponent is of class~$\mathcal{A}_{\text{global}}$, but it is neither continuous nor satisfies the decay condition. On the other hand, it does not satisfy condition~$\mathcal{N}$ either. Note that this exponent~$p$ as well as the example below in~\eqref{eq:exa1} are pointwise multipliers on $\setBMO(\RRn)$, see~\cite{Lerner05}.

However, one of the locally singular examples from Lerner's original paper~\cite{Lerner05} is defined by
\begin{align}
	p(x) = 2 - \epsilon\bigl(1 + g(x)\bigr),
\end{align}
where
\begin{align}
  \label{eq:exa1}
	g(x) =
	\begin{cases}
		\sin\left(\log\log \frac{1}{|x|}\right), & 0 < |x| \le e^{-1}, \\[1mm]
		0, & |x| > e^{-1}.
	\end{cases}
\end{align}
Because this original exponent is constant outside a bounded domain, it belongs to the class~$\mathcal{A} \cap \mathcal{N}$ for sufficiently small $\epsilon$, even though it does not have a limit at zero.

These examples already indicate that the $\log$-H\"older continuity condition might not be necessary for regularity results and a slightly weaker condition could be used.
We say that~$p \in \mathcal{A}$ if and only if
\begin{align}
	\label{eq:pinA}
	[p]_{\mathcal{A}} &\coloneqq \sup_{Q} \frac{\norm{\indicator_Q}_\px \norm{\indicator_Q}_\pdx}{\abs{Q}} < \infty,
\end{align}
where the supremum is taken over all cubes~$Q \subset \RRn$.
This condition is equivalent to the standard Muckenhoupt condition if applied to weighted Lebesgue spaces. For this reason, whenever~$p \in \mathcal{A}$, the variable exponent~$p$ is said to satisfy the Muckenhoupt condition. This condition is strictly weaker than~$\mathcal{A}_{\text{global}}$ if no additional decay assumption is used; see~\cite[Theorem~5.3.4]{DHHR} for a counterexample using bump functions separated by an exponentially growing distance.

Let us define the harmonic averages $p_Q$ by
\begin{align*}
  \frac{1}{p_Q} &\coloneqq \dashint\nolimits_Q \frac{1}{p(x)}\,dx.
\end{align*}
Then $(p')_Q = (p_Q)'$, so we can just write $p'_Q$.

It has been shown in \cite[Theorem~4.5.7]{DHHR} that $p \in \mathcal{A}$ is equivalent to
\begin{align}
  \label{eq:indicatorQ}
  \begin{aligned}
    \norm{\indicator_Q}_{\px} &\eqsim \abs{Q}^{\frac 1{p_Q}},
    \\
    \norm{\indicator_Q}_{\pdx} &\eqsim \abs{Q}^{\frac 1{p'_Q}}
  \end{aligned}
\end{align}
for all cubes~$Q$.

Kopaliani simplified in~\cite{Kopaliani} the verification of condition~$\mathcal{A}_{\text{global}}$ significantly. He showed that if $p$ is constant outside a large ball and $1<p^-\leq p^+ < \infty$, then $p \in \mathcal{A}$ already implies $p \in \mathcal{A}_{\text{global}}$. This was improved later to the fact that $p \in \mathcal{A} \cap \mathcal{N}$ and $1 < p^- \leq p^+ < \infty$ imply $p\in\mathcal{A}_{\text{global}}$ in~\cite{Lerner2010questions} and \cite[Theorem~4.52]{CUF}.  Hence, $p \in \mathcal{A} \cap \mathcal{N}$ and $1<p^-\leq p^+ < \infty$ already guarantee the boundedness of~$M$ on $L^{\px}(\RRn)$. This result was later generalized to unbounded exponents in~\cite{AdamadzeDieningKopaliani2026}.
Thus, one only needs to verify the Muckenhoupt condition $p\in\mathcal{A}$, i.e. \eqref{eq:pinA} or~\eqref{eq:indicatorQ}, and the decay condition~$p \in \mathcal{N}$.  In our regularity analysis for the PDE we will later always assume $p \in \mathcal{A} \cap \mathcal{N}$ and $1< p^- \leq p^+ < \infty$.
It has been shown in \cite[Proposition~4.7.4]{DHHR} that $p \in \mathcal{A}$ implies
\begin{align*}
  \dashint\nolimits_Q \Bigabs{ \frac 1{p(x)} - \frac 1{p_Q}}\,dx &\lesssim \frac{1}{\log(e+ \diameter(Q) + 1/\diameter(Q))}.
\end{align*}
This proves that $p \in \mathcal{A}$ implies $\frac 1p \in \setBMO^{\log}$, where $f \in \setBMO^{\log}$ if for all cubes~$Q$
\begin{align}
	\label{eq:BMOlog-est}z
	\dashint\nolimits_Q \abs{f(x) - \mean{f}_Q}\,dx &\lesssim \omega_{\log}(\diameter(Q)) = \frac{1}{\log(e + \frac{1}{
                                           \diameter(Q)})}.
\end{align}
It was shown in~\cite[Corollary~2]{Spanne65} that $C^{\log}$ is a strict subspace of $\setBMO^{\log}$, which can also be verified using the function $p$ from~\eqref{eq:example-lerner}. Thus,~\eqref{eq:BMOlog-est} does not imply that $p$ is in $C^{\log}$.  In Figure~\ref{fig:px} we summarize the relationships between different conditions on $p$ and the boundedness of the maximal function.

\begin{figure}[ht]
	\centering
	\begin{tikzpicture}[
		scale=0.75, transform shape,
		every node/.style={align=center},
		box/.style={
			draw, rounded corners, inner sep=5pt,
			minimum width=3.4cm, font=\small
		},
		arrow/.style={-{Latex[length=2mm]}, thick},
		equiv/.style={<->, >=Latex, thick},
		double slash/.style={
			-{Latex[length=2mm]}, thick,
			postaction={decorate},
			decoration={
				markings,
				mark=at position 0.5 with {
					\draw[-, thick, solid, line cap=round] (-2pt, -3pt) -- (0pt, 3pt);
					\draw[-, thick, solid, line cap=round] (1pt, -3pt) -- (3pt, 3pt);
				}
			}
		}
		]
		
		% --- Nodes ---
		\node[box] (Plog) {\textbf{Log-H\"older continuity and decay}\\ $p \in \mathcal{P}^{\log}(\mathbb{R}^n)$};
		\node[box, below=1.8cm of Plog] (Clog) { \textbf{Log-H\"older continuity and integral decay}\\ $\frac{1}{p} \in C^{\log} \cap \mathcal{N}$ };
		\node[box, below=1.8cm of Clog] (A) { \textbf{Muckenhoupt}\\ $p \in \mathcal{A}$ };
		\node[box, below=1.8cm of A] (M) { \textbf{Boundedness of $M$} };
		
		\node[box, right=2.5cm of Clog, yshift=-0.9cm] (BMO) { \textbf{Logarithmic mean  oscillation}\\ $\frac{1}{p} \in \mathrm{BMO}^{\log}$ };
		\node[box, right=2.5cm of M] (Aglobal) { \textbf{Global Muckenhoupt}\\ $p \in \mathcal{A}_{\text{global}}$ };
		
		% Plog <-> Clog (vertical parallel with double slash)
		\draw[arrow] ([xshift=+8pt]Plog.south) -- ([xshift=+8pt]Clog.north);
		\draw[double slash] ([xshift=-8pt]Clog.north) -- ([xshift=-8pt]Plog.south);
		
		% Clog <-> A (straight with +N, curved left with double slash)
		\draw[arrow] ([xshift=+8pt]Clog.south) -- ([xshift=+8pt]A.north);
		\draw[double slash] ([xshift=-8pt]A.north) --  ([xshift=-8pt]Clog.south);
		
		% A <-> M (straight with +N, curved left back)
		\draw[arrow] ([xshift=+8pt]A.south) -- node[right, font=\small] {$+\mathcal{N}$} ([xshift=+8pt]M.north);
		\draw[arrow] ([xshift=-8pt]M.north) --  ([xshift=-8pt]A.south);
		
		% Clog -> BMO
		\draw[arrow] (Clog.east) -- (BMO.north west);
		
		% A <-> BMO (parallel with double slash)
		\draw[arrow] ([yshift=-5pt]A.east) -- ([yshift=-5pt]BMO.south west);
		\draw[double slash] ([yshift=+5pt]BMO.south west) -- ([yshift=+5pt]A.east);
		
		% M <-> Aglobal
		\draw[equiv] (M.east) -- (Aglobal.west);
		
	\end{tikzpicture}
	\caption{Relations between the boundedness of $M$ and various conditions on $p$. We assume $1<p^- \leq p^+ < \infty$.} 
	\label{fig:px}
\end{figure}
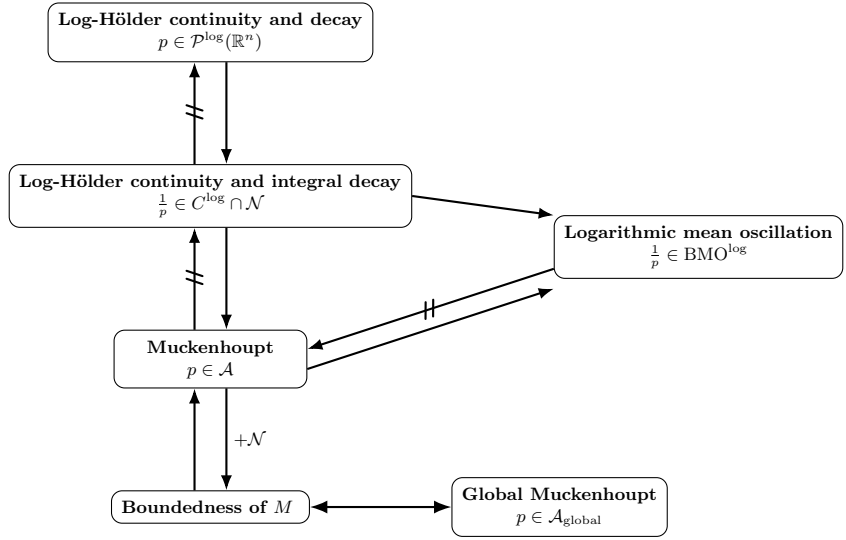

The boundedness of the maximal operator in~\cite{DHHR}, the regularity estimates for the PDE in~\cite{DieSch14} and the finite element analysis in~\cite{BreDieSch15} rely on the so-called key estimate, see~\cite[Theorem~4.2.4]{DHHR} and the refinement in~\cite{DieningSchwarzacher2013keyestimate}, which requires $p \in \mathcal{P}^{\log}(\RRn)$ though. In particular, for $p \in \mathcal{P}^{\log}(\RRn)$ and $m \in [0,\infty)$ one has the \emph{pointwise key estimate}
\begin{align}
  \label{eq:pointwise-keyestimate}
  \bigg(\dashint\nolimits_Q \abs{f(y)}\,dy\bigg)^{p(x)} \leq c_1 \dashint\nolimits_Q \abs{f(y)}^{p(y)}\,dy + c_1 \abs{Q}^m \qquad \text{for all $x \in Q$}
\end{align}
for all $f$ with $\norm{f}_\px \leq 1$.

It was shown in \cite[Lemma~4.7.3]{DHHR} that the validity of the pointwise key estimate~\eqref{eq:pointwise-keyestimate} for $p$ and $p'$ implies $\frac 1p \in C^{\log}(\RRn)$. Since we want to consider more general exponents, namely $p \in \mathcal{A} \cap \mathcal{N}$, we will weaken this pointwise key estimate and instead use an averaged version. In particular, we will show that for $p \in \mathcal{A} \cap \mathcal{N}$ we have
\begin{align}
  \label{eq:averaged-keyestimate}
  \dashint\nolimits_Q\bigg(\dashint\nolimits_Q \abs{f(y)}\,dy\bigg)^{p(x)}\,dx \leq c_1 \dashint\nolimits_Q \abs{f(y)}^{p(y)}\,dy + c_1 \abs{Q}^\epsilon
\end{align}
for sufficiently small $\epsilon>0$, see Theorem~\ref{thm:generalized-jensen}.
Due to its similarities to Jensen's inequality, we will call it~\emph{generalized Jensen's inequality}. A similar estimate for the double phase model satisfying a Muckenhoupt condition was proven in~\cite[Theorem~2.5]{adamadzedieningkopalianiok}. Before proving the generalized Jensen's inequality, we need a few auxiliary results. 

Let us define the local BMO space by $$\bignorm{\tfrac 1p}_{\setBMO(Q)} := \sup_{Q' \subseteq Q} \dashint\nolimits_{Q'} \Bigabs{ \frac 1{p(x)} - \frac 1{p_{Q'}}}\,dx.$$ 

We have the following lemma.
\begin{lemma}
  \label{lem:Qepsilon}
  Let $p \in \mathcal{A}$ with $1 < p^- \leq p^+ < \infty$. Then there exists $\epsilon=\epsilon(n,p^-,p^+,[p]_\mathcal{A})>0$ such that for all cubes~$Q$ with $\abs{Q}\leq 1$ there holds
  \begin{align}\label{eq:integral_log_holder}
    \dashint\nolimits_{Q} \abs{Q}^{-\epsilon \abs{p(x) - p_Q}}\,dx &\leq 2.
  \end{align}
\end{lemma}
\begin{proof}
First, note that if $p(x)$ is constant on $Q$ a.e. then \eqref{eq:integral_log_holder} holds trivially. Therefore, we assume that $\norm{\frac 1p}_{\setBMO(Q)} \neq 0.$ Let $\abs{Q} \leq 1$.  It follows from \cite[Proposition~4.7.4]{DHHR} that $p \in \mathcal{A}$ implies
  \begin{align*}
    \bignorm{\tfrac 1p}_{\setBMO(Q)} &\leq c_1\frac{1}{\log(e+1/\abs{Q})},
  \end{align*}
  where~$c_1$ depends only on $[p]_\mathcal{A}$ and dimension $n$. On the other hand, since $\abs{Q} \leq 1$, we have
  \begin{align*}
    \dashint\nolimits_{Q} \abs{Q}^{-\epsilon \abs{p(x) - p_Q}}\,dx
    &= \dashint\nolimits_{Q} \exp\big( \epsilon \abs{p(x) - p_Q} \abs{\log\abs{Q}}\big)\,dx
    \\
    &\leq \dashint\nolimits_{Q}\exp\big( (p^+)^2\epsilon \abs{\tfrac{1}{p_Q} - \tfrac{1}{p(x)}} \abs{\log\abs{Q}}\big)\,dx
    \\
    &\leq \dashint\nolimits_{Q}\exp\bigg( c_1 (p^+)^2\epsilon \frac{\abs{\tfrac{1}{p_Q} - \tfrac{1}{p(x)}}}{ \norm{\frac 1p}_{\setBMO(Q)}}\bigg)\,dx
  \end{align*}
Now, by \cite[Corollary~3.1.7]{Grafakos2014modern} (see also \cite{John_Nirenberg_BMO}), if $\epsilon \leq c_2 = c_2(n,c_1,p^+)$, we obtain
  \begin{align*}
    \dashint\nolimits_{Q}\abs{Q}^{-\epsilon \abs{p(x) - p_Q}}\,dx &\leq 2.
  \end{align*}
\end{proof}
Following the conventions of~\cite{DHHR}, for any $s > 0$, any cube~$Q \subset \RRn$, and any measurable $f$, we define
\begin{align}
	\label{eq:M_sQ}
	M_{s, Q} \varphi(t) &\coloneqq \left( \dashint\nolimits_Q (\varphi(x, t))^s \, dx \right)^{\frac{1}{s}}, \quad t \ge 0, \\
	M_{s, Q} \big(\varphi(f)\big) &\coloneqq \left( \dashint\nolimits_Q \big(\varphi(x, \abs{f(x)})\big)^{s} \, dx \right)^{\frac{1}{s}}.
\end{align}
Here $\varphi(x, t) \coloneqq t^{p(x)}$. When $s=1$, we just denote the preceding quantities by $M_{Q} \varphi(t)$ and $M_{Q} \big(\varphi(f)\big)$, respectively. Since we assume $1 < p^- \leq p^+ < \infty$, the convex conjugate of $\varphi(x, t) = t^{p(x)}$ satisfies $\varphi^*(x, t) \eqsim t^{p'(x)}$ (see~\cite[Definition~2.6.1]{DHHR}). For information regarding the conjugate of the form $\big(M_{s, Q} \varphi^*\big)^*(t)$ and its further properties, we refer the reader to~\cite[pp.~151--152]{DHHR}. By a slight abuse of notation, below we write $(M_{s, Q} \varphi^*)'$ to denote the derivative of the function $t \mapsto M_{s, Q} \varphi^*(t)$.

Applying Lemma~\ref{lem:Qepsilon} yields the following result, which is essential for the improved generalized Jensen's inequality. 
\begin{lemma}
	\label{lem:key-estimate-aux}
	Let $p \in \mathcal{A} \cap \mathcal{N}$ with $1<p^-\leq p^+<\infty$. Then there exist $s_0=s_0(n,p^-,p^+,[p]_{\mathcal{A}},[p]_{\mathcal{N}})>1$ and positive constants $c=c(n,p^-,p^+,[p]_{\mathcal{A}},[p]_{\mathcal{N}})$ and $\epsilon=\epsilon(n,p^-,p^+,[p]_{\mathcal{A}},[p]_{\mathcal{N}})$ such that the following holds: for all cubes~$Q$ with $\abs{Q} \leq 1$, all $s\in[1,s_0]$ and $0<t\leq 1/\norm{\indicator_Q}_\px$, we have
	\begin{align}\label{eq:key_aux_lemma}
		(M_{s,Q} \phi)(t) &\leq c\, (M_{s,Q} \phi^*)^*(t) + \abs{Q}^\epsilon \indicatorset{t\leq 1}.
	\end{align}
\end{lemma}

\begin{proof}
Let us prove \eqref{eq:key_aux_lemma} for some $s_0>1$, as formulated. The case $s\in[1,s_0]$ then follows by the classical Jensen's inequality.
For the proof we use the notation of~\cite{DHHR}, except that $\mathcal{A}_{\loc}$ and $\mathcal{A}$ from~\cite{DHHR} are replaced here by $\mathcal{A}$ and $\mathcal{A}_{\text{global}}$, respectively.  Note that $p \in \mathcal{A} \cap \mathcal{N}$ implies that $p \in \mathcal{A}_{\text{global}}$, with $[p]_{\mathcal{A}_{\text{global}}}$ depending on $p^-$, $p^+$, $[p]_{\mathcal{A}}$ and $[p]_{\mathcal{N}}$, see \cite[Theorem~4.52]{CUF}. Hence by \cite[Lemma~5.7.9]{DHHR} there exists $s_0$ such that $(M_{s_0,Q} \phi) \dominatedby (M_{s_0,Q} \phi^*)^*$ with the notation of~\cite[Definition~5.2.17]{DHHR}. We define
  \begin{align*}
    \alpha_{s_0}(Q,t) &\coloneqq \frac{(M_{s_0,Q} \phi)(t)}{(M_{s_0,Q} \phi^*)^*(t)}.
  \end{align*}
  Our goal is to show that
  \begin{align}
    \label{eq:key-estimate-aux1}
    \alpha_{s_0}(Q,t) &\lesssim 1 \qquad \text{for all } t \in \bigg[\abs{Q}^\epsilon,\frac{1}{\norm{\indicator_Q}_\px}\bigg].
  \end{align}
  Moreover, for $t \leq \abs{Q}^\epsilon$ we have by convexity
  \begin{align}
    \label{eq:key-estimate-aux2}
    (M_{s_0,Q} \phi)(t) &\leq t\, (M_{s_0,Q} \phi)(1) \leq t \leq \abs{Q}^\epsilon.
  \end{align}
  Now, the claim follows as a combination of \eqref{eq:key-estimate-aux1} and \eqref{eq:key-estimate-aux2}. It remains to prove~\eqref{eq:key-estimate-aux1}.
  
  By \cite[Lemma 5.7.16]{DHHR} we have
  \begin{align*}
    \alpha_{s_0}\bigg(Q, \frac{1}{\norm{\indicator_Q}_\px}\bigg) &\eqsim 1 \quad \text{and} \quad
    \alpha_{s_0}(Q, 1) \eqsim 1
  \end{align*}
  and
  \begin{alignat*}{2}
    \alpha_{s_0}(Q,t_2) &\lesssim \alpha_{s_0}(Q,t_1)+1 &\qquad&\text{for $0 < t_1 \leq t_2 \leq 1$},
    \\
    \alpha_{s_0}(Q,t_3) &\lesssim \alpha_{s_0}(Q,t_4)+1 &\qquad&\text{for $1\leq t_3 \leq t_4$}.
  \end{alignat*}
  This proves~\eqref{eq:key-estimate-aux1} provided that we can show that
  \begin{align}
    \label{eq:key-estimate-aux3}
    \alpha_{s_0}\big(Q, \abs{Q}^\epsilon\big) &\lesssim 1.
  \end{align}
  The proof of~\eqref{eq:key-estimate-aux3} will be based on Lemma~\ref{lem:Qepsilon}. Let us choose~$\tau \in (0,1)$ such that
  \begin{align*}
    (M_{s_0,Q} \phi^*)'(\tau) \eqsim \abs{Q}^\epsilon.
  \end{align*}
  Then $\abs{Q}^\epsilon \lesssim \tau^{(p^+)'-1}$, which implies
  \begin{align}
    \label{eq:key-estimate-aux4}
    \tau \gtrsim  \abs{Q}^{\epsilon(p^+-1)}.
  \end{align}
  The implicit constants in the above estimates depend only on $p^-$ and $p^+$. Moreover, it follows by the calculations in \cite[Lemma~5.7.16]{DHHR} and Jensen's inequality that
  \begin{align*}
     \alpha_{s_0}(Q, \abs{Q}^\epsilon) &\eqsim \alpha_{s_0}\big(Q, (M_{s_0,Q} \phi^*)'(\tau)\big) 
     \\
     &\eqsim \Bigg( \dashint\nolimits_Q \bigg( \dashint\nolimits_Q \tau^{\frac{s_0(p(y)-p(z))}{(p(y)-1)(p(z)-1)}} \,dz \bigg)^{p(y)-1} \,dy \Bigg)^{\frac{1}{s_0}}
     \\
     &\lesssim \Bigg( \dashint\nolimits_Q \bigg( \dashint\nolimits_Q \tau^{\frac{-s_0\abs{p(y)-p(z)}}{(p(y)-1)(p(z)-1)}} \,dz \bigg)^{p(y)-1} \,dy \Bigg)^{\frac{1}{s_0}}
     \\
     &\lesssim \Bigg( \dashint\nolimits_Q  \dashint\nolimits_Q \tau^{\frac{-s_0 p^+\abs{p(y)-p(z)}}{(p(y)-1)(p(z)-1)}} \,dz \,dy \Bigg)^{\frac{1}{s_0}}.
  \end{align*}
  Setting $p_c\coloneq \big(\frac{p^+}{p^- -1}\big)^2$, we obtain
  \begin{align*}
    \alpha_{s_0}(Q,\abs{Q}^\epsilon) &\lesssim \bigg( \dashint\nolimits_Q \dashint\nolimits_Q \tau^{\frac{-s_0 p^+ \abs{p(y)-p(z)}}{(p^- -1)^2}} \,dz \,dy \Bigg)^{\frac{1}{s_0}}
    \\
    &
    \lesssim \bigg( \dashint\nolimits_Q \dashint\nolimits_Q \abs{Q}^{-s_0 \epsilon p_c \abs{p(y)-p(z)}} \,dz \,dy \Bigg)^{\frac{1}{s_0}}
    \\
    &\lesssim \bigg( \dashint\nolimits_Q \abs{Q}^{-s_0 \epsilon p_c \abs{p(y)-p_Q}}\,dy  \dashint\nolimits_Q \abs{Q}^{-s_0 \epsilon p_c \abs{p(z)-p_Q}}\,dz \Bigg)^{\frac{1}{s_0}}.
  \end{align*}
  Now, for $\epsilon$ small enough depending on $s_0$, $p^-$, $p^+$ and $[p]_{\mathcal{A}}$ it follows by Lemma~\ref{lem:Qepsilon} that $\alpha_{s_0}(Q,\abs{Q}^\epsilon) \lesssim 1$. This proves~\eqref{eq:key-estimate-aux3}. The proof of the lemma is complete.
\end{proof}
Using the preceding lemma, we obtain the following generalized Jensen's inequality.
In what follows, we consider the sum of the two spaces and its associated norm, defined as:
\begin{align*}
	L^{r(\cdot)}(\mathbb{R}^n) + L^{s(\cdot)}(\mathbb{R}^n) &= \big\{f = g + h : g \in L^{r(\cdot)}(\mathbb{R}^n), h \in L^{s(\cdot)}(\mathbb{R}^n)\big\}, \\
	\intertext{where the norm is given by}
	\|f\|_{L^{r(\cdot)}+L^{s(\cdot)}} &= \inf_{\{f=g+h, g\in L^{r(\cdot)}, h\in L^{s(\cdot)}\}} (\|g\|_{r(\cdot)} + \|h\|_{s(\cdot)}).
\end{align*}
\begin{theorem}[Generalized Jensen's inequality]
	\label{thm:generalized-jensen}
	Let $p \in \mathcal{A} \cap \mathcal{N}$ with $1<p^-\leq p^+<\infty$. Then there exist constants $c=c(n,p^-,p^+,[p]_{\mathcal{A}},[p]_{\mathcal{N}})>0$ and $\epsilon=\epsilon(n,p^-,p^+,[p]_{\mathcal{A}},[p]_{\mathcal{N}})>0$ such that the following holds: for all cubes~$Q$ with $\abs{Q} \leq 1$ and $\norm{f}_{L^{\px} + L^\infty} \leq 1$,
	\begin{align*}
		\dashint\nolimits_Q \bigg( \dashint\nolimits_Q \abs{f(y)}\,dy \bigg)^{p(x)} \,dx  &\leq c\, \dashint\nolimits_Q \abs{f(x)}^{p(x)}\,dx  + \abs{Q}^\epsilon. 
	\end{align*}
\end{theorem}
\begin{proof}
	Putting $s=1$ in \eqref{eq:key_aux_lemma}, we get that
	\[(M_{Q} \phi)(t)\leq c\, (M_{Q}\phi^*)^*(t) + \abs{Q}^\epsilon \indicatorset{t\leq 1}.\]
	On the other hand, due to \cite[Lemma 5.2.8]{DHHR}, we have that $(M_{Q}\phi^*)^*(M_Qf)\leq M_Q(\phi(f))$. Putting this together gives us the desired result. The only thing left to check is that $\norm{f}_{L^{\px} + L^\infty} \leq 1$ implies that $M_Qf\lesssim \frac{1}{\norm{\indicator_Q}_{\px}}$, and using the $\Delta_2$ condition, we can use \eqref{eq:key_aux_lemma} for $t=M_Qf$.
	
	By the triangle inequality, it suffices to consider the cases $\norm{f}_{\px} \leq 1$ and $\norm{f}_\infty \leq 1$ separately. 
	If $\norm{f}_{\px} \leq 1$, then, with $p \in \mathcal{A}$ and $\abs{Q} \leq 1$,
	\begin{align*}
		M_Q f \leq 2\,\frac{\norm{\indicator_Q}_{\pdx}}{\abs{Q}} \lesssim \frac{1}{\norm{\indicator_Q}_{\px}}.
	\end{align*}
	If $\norm{f}_\infty \leq 1$, then, with $\abs{Q} \leq 1$,
	\begin{align*}
		M_Q f \leq 1 \lesssim \frac{1}{\norm{\indicator_Q}_{\px}}.
	\end{align*}
	This completes the proof.
\end{proof}

Although the preceding estimate is of independent interest, we require the improved generalized Jensen's inequality. This inequality is essential for the Sobolev--\Poincare inequality, which, in turn, is a key ingredient for our regularity results. Furthermore, note that Theorem~\ref{thm:generalized-jensen} follows from the next theorem by setting $s=1$.

\begin{theorem}[Improved generalized Jensen's inequality]
  \label{thm:improved-jensen}% 
Let $p \in \mathcal{A} \cap \mathcal{N}$ with $1<p^-\leq p^+<\infty$. Then there exist constant $s_0=s_0(n,p^-,p^+,[p]_{\mathcal{A}},[p]_{\mathcal{N}})>1$, $c=c(n,p^-,p^+,[p]_{\mathcal{A}},[p]_{\mathcal{N}})>0$ and $\epsilon=\epsilon(n,p^-,p^+,[p]_{\mathcal{A}},[p]_{\mathcal{N}})>0$ such that the following holds: for all $s\in[1,s_0]$, all cubes~$Q$ with~$\abs{Q} \leq 1$ and $\norm{f}_{L^{\px/s} + L^{\px} + L^\infty} \leq 1$, we have
  \begin{align}\label{eq:improved_Jens}
    (M_{s,Q} \phi)(M_Q f) &\leq c\, M_{1/s,Q} \big(\phi(f)\big) + \abs{Q}^\epsilon.
  \end{align}
  In other words,
  \begin{align*}
    \Bigg(\dashint\nolimits_Q \bigg( \dashint\nolimits_Q \abs{f(y)}\,dy \bigg)^{sp(x)} \,dx \Bigg)^{\frac 1s} &\leq c\, \bigg( \dashint\nolimits_Q \abs{f(x)}^{p(x)/s}\,dx \bigg)^{s} + \abs{Q}^\epsilon. 
  \end{align*}
\end{theorem}
\begin{proof}
We prove \eqref{eq:improved_Jens} for $s_0>1$; the case $s\in[1,s_0]$ follows from Jensen's inequality.
 Since $p \in \mathcal{A} \cap \mathcal{N}$, we have $p \in \mathcal{A}_{\text{global}}$, with $[p]_{\mathcal{A}_{\text{global}}}$ depending on $p^-$, $p^+$, $[p]_{\mathcal{A}}$ and $[p]_{\mathcal{N}}$, see \cite[Theorem~4.52]{CUF}. Hence, by \cite[Theorem~5.7.2]{DHHR} we find $\theta_{0} \in (1/p^-,1)$ with $\theta_0 = \theta_0(p^-,p^+,[p]_{\mathcal{A}},[p]_\mathcal{N})$ such that $\theta_0 p \in \mathcal{A}_{\text{global}}$. Thus $\theta_0 p \in \mathcal{N} \cap \mathcal{A}$.  Hence, we can apply Lemma~\ref{lem:key-estimate-aux} to $\theta_0 p$. Due to \cite[Lemma 4.4.7]{DHHR} we have that $(\theta p)^-$, $(\theta p)^+$, $[\theta p]_{\mathcal{A}}$ and $[\theta p]_{\mathcal{N}}$ are uniformly bounded with respect to $\theta \in [\theta_0,1)$. Thus, we can find $\tilde{s} = \tilde{s}(n,p^-,p^+,[p]_{\mathcal{A}}, [p]_{\mathcal{N}})>1$ and $\epsilon=\epsilon(n,p^-,p^+,[p]_{\mathcal{A}}, [p]_{\mathcal{N}})>0$ such that we can apply Lemma~\ref{lem:key-estimate-aux} to $\theta p$ with $\tilde{s}$ and~$\epsilon$.

  Now, let us fix $\theta \in [\theta_0,1)$ and $s_0 \in (1,\sqrt{\tilde{s}}]$ such that $\theta s_0 = 1$. Let $\rho(x,t) = t^{p(x)/s_0} = \phi(x,t^{1/s_0}) = \phi^{1/s_0}(x,t)$. In particular, it follows from Lemma~\ref{lem:key-estimate-aux} using $s_0^2 \in (1,\tilde{s}]$
  that for $\abs{Q} \leq 1$ and $t \leq 1/\norm{\indicator_Q}_{\theta \px}$, there holds
  \begin{align}
    \label{eq:key-estimate-final1}
    (M_{s_0^2,Q} \rho)(t) &\leq c\, (M_{s_0^2,Q} \rho^*)^*(t) + \abs{Q}^\epsilon \indicatorset{t\leq 1}.
  \end{align}
  We will prove the theorem using this choice of~$s_0>1$ and $\theta<1$.  Proceeding exactly as in the preceding theorem, we consider the cases $\norm{f}_{\theta \px} \leq 1$, $\norm{f}_{\px} \leq 1$, and $\norm{f}_\infty \leq 1$ separately. If $\norm{f}_{\theta \px} \leq 1$, then with $\theta p \in \mathcal{A}$
  \begin{align*}
    M_Q f \leq 2\,\frac{\norm{\indicator_Q}_{(\theta \px)'}}{\abs{Q}} \lesssim \frac{1}{\norm{\indicator_Q}_{\theta\px}}.
  \end{align*}
  If $\norm{f}_{\px} \leq 1$, then with $p \in \mathcal{A}$ and $\abs{Q} \leq 1$
  \begin{align*}
    M_Q f \leq 2\,\frac{\norm{\indicator_Q}_{\pdx}}{\abs{Q}} \lesssim \frac{1}{\norm{\indicator_Q}_{\px}} = \frac{1}{\norm{\indicator_Q}_{\theta\px}^\theta} \lesssim \frac{1}{\norm{\indicator_Q}_{\theta\px}}.
  \end{align*}
  If $\norm{f}_\infty \leq 1$, then with $\abs{Q} \leq 1$
  \begin{align*}
    M_Q f \leq 1 \lesssim \frac{1}{\norm{\indicator_Q}_{\theta\px}}.
  \end{align*}
  In all cases $M_Q f \lesssim \frac{1}{\norm{\indicator_Q}_{\theta\px}}$. Using the $\Delta_2$-condition, it suffices to assume in the following that
  \begin{align}
    \label{eq:key-estimate-final2}
    M_Q f \leq \frac{1}{\norm{\indicator_Q}_{\theta\px}}.
  \end{align}
  Hence, we can apply~\eqref{eq:key-estimate-final1} to $t=M_Q f$ to obtain
  \begin{align*}
    (M_{s_0^2,Q} \rho)(M_Q f) &\leq c\, (M_{s_0^2,Q} \rho^*)^*(M_Q f) + \abs{Q}^\epsilon \indicatorset{M_Qf\leq 1}.
  \end{align*}
  By Jensen's inequality we have $(M_{s_0^2,Q} \rho^*)^* \!\leq\! (M_Q \rho^*)^*$. Moreover, $(M_Q \rho^*)^*(M_Q f) \!\le\! M_Q(\rho(f))$ by \cite[Lemma~5.2.8]{DHHR}, so
  \begin{align*}
    (M_{s_0^2,Q} \rho)(M_Q f) &\leq c\, M_Q \big(\rho(f)\big) + \abs{Q}^\epsilon.
  \end{align*}
  Using $\rho^{s_0}(x,t) = \phi(x,t)$ and $s_0\theta =1$, we can rewrite this as
  \begin{align*}
    \big((M_{s_0,Q} \phi)(M_Q f)\big)^{\frac{1}{s_0}} &\leq c\, \big(M_{\theta,Q} \big(\phi(f)\big) \big)^{\frac{1}{s_0}} + \abs{Q}^\epsilon.
  \end{align*}
  Hence,
  \begin{align*}
    (M_{s_0,Q} \phi)(M_Q f) &\lesssim M_{\theta,Q} \big(\phi(f)\big) + \abs{Q}^{s_0\epsilon} \leq M_{\theta,Q} \big(\phi(f)\big) + \abs{Q}^{\epsilon}.
  \end{align*}
  This proves the theorem.
\end{proof}
\begin{remark}\label{rem:normilized_cubes}
  Note that in Theorem~\ref{thm:improved-jensen} in the size restriction $\abs{Q} \leq 1$ as well as in the norm restriction $\norm{f}_{L^{p(\cdot)/s} + L^{p(\cdot)} + L^\infty} \leq 1$ the constant~$1$ can be replaced by any fixed $L>0$.  The constant in the final estimate then depends on $L$. The same applies to all subsequent estimates, and the generalized Jensen inequality can be used for both cubes and balls. This allows us to formulate subsequent results under the normalizations $|Q|\leq 1$ and $|B|\leq 1$. Even if a specific estimate requires a smaller upper bound on the measure of the cubes (or balls), this observation lets us maintain these normalizations at the expense of an additional dependence in the constants.
\end{remark}
\section{Sobolev--\Poincare Inequality}
\label{sec:Sobolev-Poincare}
In this section, we prove a novel Sobolev--\Poincare{}-type inequality for variable exponents satisfying the Muckenhoupt condition coupled with a decay condition. The main ingredient is the improved generalized Jensen's inequality of Theorem~\ref{thm:improved-jensen}. Obtaining the Sobolev--\Poincare{} inequality using a relevant potential estimate is a standard technique. However, the presence of the error term in our improved generalized Jensen's inequality necessitates some adjustments. For the sake of completeness and to clarify the impact of this error term, we provide a proof below. First of all, note that Theorem~\ref{thm:generalized-jensen} and Theorem~\ref{thm:improved-jensen} remain valid if we replace cubes with balls. We assume without loss of generality that $p^- \leq n$, since $p^-$ is for us just a lower bound of~$p$. Note that for $p^- > n$ the Sobolev embedding theorem guarantees that weak solutions are H\"{o}lder continuous. However, the explicit $L^\infty$ estimates are also interesting for $p^- > n$.

Let us introduce the averaging operator.
For $k\in \mathbb{Z}$, we define the averaging operator over dyadic cubes by
\begin{align*}
	T_kf := \sum_{\substack{Q \text{ dyadic}\\ \ell(Q)=2^{-k}}} \indicator_Q \dashint\nolimits_{2Q}\abs{f}\,dy,
\end{align*}
where $\ell(Q)$ is the sidelength of $Q$. If $B$ is a cube, $r_B$ is its side length; if a ball, its radius. Recall that a dyadic cube in $\mathbb{R}^n$ takes the form $\prod_{i=1}^n [2^j m_i, 2^j(m_i + 1))$ for $j, m_1, \dots, m_n \in \mathbb{Z}$. Note that $\rho_{\px}(\abs{f})\coloneq\int_\Omega\abs{f(x)}^{p(x)}\,dx$.
\begin{theorem}[Sobolev--\Poincare]
	\label{thm:improved_poincare}%
Let $p \in \mathcal{A} \cap \mathcal{N}$ with $1 < p^- \leq p^+ < \infty$. Then there exists a constant $s_0=s_0(n,p^-,p^+,[p]_{\mathcal{A}},[p]_{\mathcal{N}})$ with $1 < s_0 < \min\lbrace p^-,n'\rbrace$, and positive real numbers $c=c(n,p^-,p^+,[p]_{\mathcal{A}},[p]_{\mathcal{N}})$ and $\epsilon=\epsilon(n,p^-,p^+,[p]_{\mathcal{A}},[p]_{\mathcal{N}})$ such that the following holds: for all balls~$B$ with~$\abs{B} \leq 1$, all $s_1,s_2 \in [1,s_0]$, and all $u\in W^{1,1}(B)$ with $\norm{\nabla u}_{\px/s_2} \leq 1$,
	\begin{align}
    \label{eq:improved_poincare}
		\bigg(\dashint\nolimits_B \bigg( \frac{\abs{u - \mean{u}_B}}{r_B}\bigg)^{s_1 p(x)}\,dx \Bigg)^{\frac 1{s_1}}
		&\lesssim
		\bigg( \dashint\nolimits_B \abs{\nabla u}^{p(x)/{s_2}}\,dx \bigg)^{s_2} + \abs{B}^\epsilon.
	\end{align}
    It is possible to replace $\mean{u}_B$ by $\mean{u}_E$, where $E \subset B$ with $\abs{E} \eqsim \abs{B}$. Moreover, $B$ can be a ball or a cube.
\end{theorem}
\begin{proof}
First, we consider the case where $s=s_1=s_2$ in \eqref{eq:improved_poincare}. To obtain \eqref{eq:improved_poincare} for any $s_1,s_2 \in [1,s]$, one can apply the classical Jensen's inequality to the first case. We use the following standard estimate; see \cite[Lemma 1.50]{Maly} and \cite[Lemma 6.1.4]{DHHR}
  \begin{align}
    \label{eq:maly-ziemer}
    \abs{u(x) - \mean{u}_B} \lesssim \int\nolimits_B \frac{|\nabla u(y)|}{|x - y|^{n-1}} \,dy \lesssim r_B \sum_{k=0}^{\infty} 2^{-k} T_{k+k_0} (\indicator_B|\nabla u|)(x)
  \end{align}
  for all $x\in B$, where $k_0\in\mathbb{Z}$ is such that $2^{-k_0-1}\leq r_B\leq2^{-k_0}$. Consequently, because we only deal with cubes $Q$ of side length $2^{-k}$ for $k \geq k_0$ in what follows, we have $\abs{2Q}\lesssim1$. Remark~\ref{rem:normilized_cubes} justifies the use of the improved generalized Jensen's inequality below for these cubes. Using this we get that
  \begin{align*}
    |u(x)-\mean{u}_B| \lesssim r_B
    \sum_{k=k_0}^{\infty}2^{-k+k_0} \sum_{\substack{Q \, \text{dyadic} \\ \ell(Q)=2^{-k}}} \indicator_{Q}(x) \dashint\nolimits_{2Q} \indicator_{ B}|\nabla u|\,dy.
  \end{align*}
  Below, we only use dyadic cubes, and for convenience, we omit writing that $Q$ is dyadic. Using the preceding inequality, the fact that $\ell^1 \hookrightarrow \ell^s$ for $s > 1$, and Theorem~\ref{thm:improved-jensen}, we proceed as follows:
	\begin{align*}
		\lefteqn{\Bigg(\dashint\nolimits_{B}\Big|\frac{u-\mean{u}_B}{r_B}\Big|^{sp(x)}dx\Bigg)^{\frac{1}{s}}}\qquad&
		\\
		&\lesssim\Bigg(\dashint\nolimits_{B}\Big(\sum_{\substack{k\geq k_0}}2^{-(k-k_0)}\sum_{\substack{\ell(Q)=2^{-k}}}\indicator_{Q}(x)\dashint\nolimits_{2Q}\indicator_{B}(y)|\nabla u(y)|dy \Big)^{sp(x)}dx\Bigg)^{\frac{1}{s}}
		\\
		&\lesssim\sum_{\substack{k\geq k_0}}2^{-(k-k_0)}\Bigg(\dashint\nolimits_{B}\Big(\sum_{\substack{\ell(Q)=2^{-k}}} \indicator_{Q}(x)\dashint\nolimits_{2Q}\indicator_{B}(y)\abs{\nabla u(y)}dy\Big)^{sp(x)}dx\Bigg)^{\frac{1}{s}}
		\\
		&\lesssim\sum_{\substack{k\geq k_0}}2^{-(k-k_0)}\sum_{\substack{\ell(Q)=2^{-k}\\B\cap Q\neq\emptyset}} \Bigg(\frac{\abs{Q}}{|B|}\Bigg)^{\frac{1}{s}}\Bigg(\dashint\nolimits_{2Q}\Big(\dashint\nolimits_{2Q}\indicator_B(y)\abs{\nabla u(y)}dy\Big)^{sp(x)}dx\Bigg)^{\frac{1}{s}}
		\\
		&\lesssim\sum_{\substack{k\geq k_0}}2^{-(k-k_0)}\sum_{\substack{\ell(Q)=2^{-k}\\B\cap Q\neq\emptyset}}\Bigg(\frac{\abs{Q}}{|B|}\Bigg)^{\frac{1}{s}} \Biggl(\Bigg(\dashint\nolimits_{2Q}\indicator_B\abs{\nabla u}^{\frac{p(x)}{s}}\,dx \Bigg)^{s}+\abs{2Q}^{\epsilon}\Biggr)=:I
	\end{align*}
Choosing $s$ close to $1$ such that $n \left( s - \frac{1}{s} \right) < 1$ and verifying the hypotheses of Theorem~\ref{thm:improved-jensen} using the fact that $\ell^1 \hookrightarrow \ell^s$ for $s > 1$, we get
\begin{align*}
I&\lesssim\sum_{\substack{k\geq k_0}}2^{-(k-k_0)}\sum_{\substack{\ell(Q)=2^{-k}\\B\cap Q\neq\emptyset}}\Bigg(\frac{\abs{Q}}{|B|}\Bigg)^{\frac{1}{s}-s}\Bigg(\dashint\nolimits_{B}\indicator_{2Q}\abs{\nabla u}^{\frac{p(x)}{s}}\,dx\Bigg)^s
		\\
    &+\sum_{\substack{k\geq k_0}}2^{-(k-k_0)}\sum_{\substack{\ell(Q)=2^{-k}\\B\cap Q\neq\emptyset}}\Bigg(\frac{\abs{Q}}{|B|}\Bigg)^{\frac{1}{s}}\abs{Q}^{\epsilon}
		\\
		&\lesssim\bigg(\dashint\nolimits_B\abs{\nabla u}^{p(x)/s}\,dx \bigg)^s + \abs{B}^\epsilon.
 	\end{align*}
    This proves the claim using $\mean{u}_B$. The same calculation also holds with $\mean{u}_B$ replaced by $\mean{u}_E$, since the first estimate in~\eqref{eq:maly-ziemer} remains valid as explained in detail in~\cite[Remark~8.2.10]{DHHR}. The rest of the calculation is then unchanged.
\end{proof}
We will apply Theorem~\ref{thm:improved_poincare} in particular to functions $u$ with $u=0$ on a set $E\subset B$ with $\abs{E} \eqsim \abs{B}$. In that case we can drop the $\mean{u}_E$ in~\eqref{eq:improved_poincare} completely. The error term~$\abs{B}^\epsilon$, however, causes difficulties later in our De~Giorgi iteration, since it is not small for $u$ small. Therefore, we need an improved error term. This is the purpose of the following lemma. The proof is based on a careful Calder\'on-Zygmund stopping-time argument. For the remainder of this section, we assume that $B$ is a cube.

\begin{lemma}\label{lem:stop}
Let $p \in \mathcal{A} \cap \mathcal{N}$ with $1<p^- \leq p^+ < \infty$. Then there exists a constant $s=s(n,p^-,p^+,[p]_{\mathcal{A}},[p]_{\mathcal{N}})$ with $1<s<\min\{s_0,\sqrt{n'},n/(n+1-p^-)\}$, where $s_0$ is the same as in the preceding theorem, and a positive real number $\epsilon=\epsilon(n,p^-,p^+,[p]_{\mathcal{A}},[p]_{\mathcal{N}},s)$ such that the following holds: for all cubes~$B$ with~$\abs{B}\leq1$, all $s_1,s_2 \in [1,s]$, and all $u\in W^{1,1}(B)$ with $\norm{\nabla u}_{\px/s_2} \leq 1$ vanishing on a set $E \subset B$ with $\abs{E}\eqsim\abs{B}$, we have
	\begin{align}\label{eq:stopping_argument}
	\Bigg(\dashint\nolimits_B \bigg( \frac{\abs{u}}{r_B}\bigg)^{s_1 p(x)}\,dx \Bigg)^{\frac 1{s_1}}
    &\lesssim
	\bigg( \dashint\nolimits_B \abs{\nabla u}^{p(x)/s_2}\,dx \bigg)^{s_2} +\abs{B}^{\epsilon} \Bigg(\frac{\abs{B\cap\set{\abs{u}>0}}}{\abs{B}}\Bigg)^{1+\frac{1}{n}}.
  \end{align}
\end{lemma} 
\begin{proof}	
Let us consider two cases. The first is when 
\begin{align}\label{eq:stop_1}
	\frac{\abs{B\cap\{\abs{u}>0\}}}{\abs{B}}\ge \frac{1}{2^{n+2}},
\end{align}
and the second is when 
\begin{align}\label{eq:stop_2}
	\frac{\abs{B\cap\{\abs{u}>0\}}}{\abs{B}}< \frac{1}{2^{n+2}}.
\end{align}
Note that if \eqref{eq:stop_1} holds, then the desired result directly follows from Theorem~\ref{thm:improved_poincare}. The interesting case is when we have \eqref{eq:stop_2}. In this case, we proceed as follows.
 We start by bisecting the cube $B$ into $2^n$ congruent subcubes. If a dyadic subcube $Q$ satisfies
  \begin{align*}
    \frac{|Q\cap\{\abs{u}>0\}|}{|Q|}\ge \frac{1}{2^{n+2}},
  \end{align*}
  we select and stop on it. If not, we bisect it into $2^{n}$ dyadic children and continue recursively. Let $S_k$ be the family of cubes selected at step $k$, and let $R_k$ be the family of cubes which remain unselected after $k$ steps. For each positive integer $m$ ($m\in\mathbb{N}$) we define $\mathcal{T}_m$ as the union of these disjoint cubes, i.e.,
  \begin{align*}
    \mathcal T_m:=\Big(\bigcup_{k=1}^m S_k\Big)\cup R_m.
  \end{align*} It is clear that $\mathcal{T}_m$ is a partition of $B$ (up to null sets) for each $m$. If $Q\in R_m$, then by construction
  \begin{align*}
    \frac{|Q\cap\{\abs{u}>0\}|}{|Q|}<\frac{1}{2^{n+2}},
  \end{align*}
  and $|Q\cap\{u=0\}|\geq(1-\frac{1}{2^{n+2}})|Q|\eqsim |Q|.$	
  If $Q\in S_k$, then $Q$ is selected, so $$\frac{|Q\cap\{\abs{u}>0\}|}{|Q|}\geq \frac{1}{2^{n+2}}.$$ On the other hand, its parent was not selected. Let us denote it by $\widetilde{Q}$, which means that $\frac{|\widetilde Q\cap\{\abs{u}>0\}|}{|\widetilde Q|}<\frac{1}{2^{n+2}}.$
  Since $|\widetilde Q|=2^n|Q|$ and $Q\subset \widetilde Q$, we obtain
  \begin{align*}
    \frac{|Q\cap\{\abs{u}>0\}|}{|Q|}\leq\frac{|\widetilde Q\cap\{\abs{u}>0\}|}{|Q|}=2^n\frac{|\widetilde Q\cap\{\abs{u}>0\}|}{|\widetilde Q|}<\frac{1}{4}.
  \end{align*}
  Therefore
  \begin{align*}
    |Q\cap\{u=0\}|\geq \tfrac{3}{4}|Q|.
  \end{align*}
  Thus every cube $Q\in\mathcal T_m$ satisfies $|Q\cap\{u=0\}|\eqsim|Q|.$ As we already mentioned, this allows us to use \eqref{eq:improved_poincare} for every $ Q\in\mathcal{T}_m$ without subtracting the average on the left-hand side.
  It may happen that the union $\bigcup_{k=1}^m S_k$ is empty for some $m$, in which case below we use the standard convention that a sum over an empty index set is zero. Let us denote $\delta:=sp^-+n(1-s^2).$ Then for each $m\in\mathbb{N}$ we have the following:
  \begin{align*}
    \lefteqn{\dashint\nolimits_B \Big(\frac{|u|}{r_B}\Big)^{sp(x)}\,dx}\qquad&
    \\
    &=\sum_{Q\in\mathcal T_m}\frac{|Q|}{|B|}\dashint\nolimits_Q\Big(\frac{|u|}{r_Q}\Big)^{sp(x)}\Big(\frac{r_Q}{r_B}\Big)^{sp(x)}\,dx
    \\
    &\lesssim\sum_{Q\in \cup_{k=1}^m S_k}\frac{|Q|}{|B|}\Big(\frac{r_Q}{r_B}\Big)^{sp^-}
      \!\!\dashint\nolimits_Q \Big(\frac{|u|}{r_Q}\Big)^{sp(x)}\,dx
      +\sum_{Q\in R_m}\frac{|Q|}{|B|}\Big(\frac{r_Q}{r_B}\Big)^{sp^-}\!\!\dashint\nolimits_Q \Big(\frac{|u|}{r_Q}\Big)^{sp(x)}\,dx
    \\
    &\lesssim\sum_{Q\in \cup_{k=1}^m S_k}\frac{|Q|}{|B|}\Big(\frac{r_Q}{r_B}\Big)^{sp^-}\Bigg(\bigg(\dashint\nolimits_Q \abs{\nabla u}^{p(x)/s}\,dx\bigg)^{s^2}+|Q|^{\varepsilon s}\Bigg)
    \\
    &\quad +2^{-msp^-}\sum_{Q\in R_m}\frac{|Q|}{|B|}\Bigg(\bigg(\dashint\nolimits_Q \abs{\nabla u}^{p(x)/s}\,dx\bigg)^{s^2}+|Q|^{\epsilon s}\Bigg)
    \\
    &\lesssim|B|^{-1-\frac{sp^-}{n}} \!\!\!\!\!\sum_{Q\in \cup_{k=1}^m S_k} \!\!\!\!\!|Q|^{1+\frac{sp^-}{n}-s^2}\bigg(\int\nolimits_Q \abs{\nabla u}^{p(x)/s}\,dx\bigg)^{s^2}
      \!\!+ \!\!\!\sum_{Q\in \cup_{k=1}^m  S_k} \!\!\!\!\frac{|Q|}{|B|}\Big(\frac{r_Q}{r_B}\Big)^{sp^-}|Q|^{\epsilon s}
    \\
    &\quad+2^{-msp^-}\sum_{Q\in R_m}\frac{|Q|}{|B|}\bigg(\dashint\nolimits_Q \abs{\nabla u}^{p(x)/s}\,dx\bigg)^{s^2}
      +2^{-msp^-}\sum_{Q\in R_m}\frac{|Q|}{|B|}|Q|^{\varepsilon s}
    \\
    &\leq|B|^{-s^2}\sum_{Q\in \cup_{k=1}^m S_k}\bigg(\int\nolimits_Q \abs{\nabla u}^{p(x)/s}\,dx\bigg)^{s^2}+\abs{B}^{\epsilon s}\left(\sum_{Q\in \cup_{k=1}^m S_k}\Big(\frac{|Q|}{|B|}\Big)^{1+\frac{sp^-}{n}}\right)
    \\
    &\quad +2^{-msp^-}|B|^{-1}(2^{-mn}|B|)^{1-s^2}
      \sum_{Q\in R_m}\bigg(\int\nolimits_Q \abs{\nabla u}^{p(x)/s}\,dx\bigg)^{s^2}+2^{-msp^-}
    \\
    &\lesssim \bigg(\fint\nolimits_B \abs{\nabla u}^{p(x)/s}\,dx\bigg)^{s^2}+ \abs{B}^{\epsilon s} \bigg(\sum_{Q\in \cup_{k=1}^m S_k}\frac{|Q|}{|B|}\bigg)^{1+\frac{sp^-}{n}}
    \\
      &\quad+2^{-m\delta}\bigg(\fint\nolimits_B \abs{\nabla u}^{p(x)/s}\,dx\bigg)^{s^2}+2^{-msp^-}
    \\
    &\lesssim\left(\dashint\nolimits_B |\nabla u|^{p(x)/s}\,dx\right)^{s^2}+\abs{B}^{\epsilon s} \left(\frac{|B\cap\{\abs{u}>0\}|}{|B|}\right)^{s(1+\frac1n)}
    \\
     &\quad+2^{-m\delta}\bigg(\dashint\nolimits_B \abs{\nabla u}^{p(x)/s}\,dx\bigg)^{s^2}+2^{-msp^-}.
  \end{align*}
Letting $m \to \infty$ gives \eqref{eq:stopping_argument} in the case $s_1 = s_2 = s$. Having obtained the same powers, the desired result follows by using classical Jensen's inequality.
\end{proof}
Using the zero-extension technique, we can obtain the following:

\begin{corollary}\label{cor:zero_boundary_value}
Let $p$, $s$, and $\epsilon$ be as in Lemma~\ref{lem:stop}. Then for all cubes~$B$ with~$\abs{B}\leq1$, all $s_1, s_2 \in [1, s]$, and all $u\in W_0^{1,1}(B)$ with $\norm{\nabla u}_{\px/s_2} \leq 1$, the following holds:
	\begin{equation}\label{eq:zero_boundary_coro}
		\left( \dashint\nolimits_{B} \left(\frac{\abs{u}}{r_B}\right)^{s_1 p(x)}\,dx \right)^{\frac{1}{s_1}} \lesssim \left(\dashint\nolimits_{B} \abs{\nabla u}^{p(x)/s_2}\,dx \right)^{s_2} + \abs{B}^\epsilon \left(\frac{\abs{B \cap \{\abs{u}>0\}}}{\abs{B}}\right)^{1+\frac{1}{n}}.
	\end{equation}
\end{corollary}
\begin{remark}\label{rem:smallness_of_norms}
  Note that Theorems \ref{thm:generalized-jensen}, \ref{thm:improved-jensen}, and \ref{thm:improved_poincare}, Lemma \ref{lem:stop}, and Corollary \ref{cor:zero_boundary_value} all remain valid if, in the corresponding smallness assumptions on the norms and the cube (or ball), one replaces $1$ with any constant $L$, compare with Remark~\ref{rem:normilized_cubes}. In this case, the constants in the corresponding estimates will also depend on $L$. Moreover, all these results remain true for both cubes and balls.
\end{remark}
\section{Regularity Results}
\label{sec:Regularities}
A function $u \in W^{1,p(\cdot)}(\Omega)$ is a weak solution of \eqref{eq:px-laplacian_non_zero} if 
\begin{equation}\label{eq:weak_form}
	\int\nolimits_{\Omega} \abs{\nabla u}^{p(x)-2}\nabla u \cdot \nabla \psi \, dx = \int\nolimits_{\Omega} \abs{G}^{p(x)-2}G \cdot \nabla \psi \, dx \quad \forall \psi \in W_0^{1,p(\cdot)}(\Omega).
\end{equation}
Furthermore, $u$ is a weak subsolution (resp. supersolution) of \eqref{eq:px-laplacian_non_zero} if \eqref{eq:weak_form} holds with $\le$ (resp. $\ge$) in place of $=$ for all non-negative $\psi \in W_0^{1,p(\cdot)}(\Omega)$. 
In the case of a zero right-hand side, $u$ is a weak solution of \eqref{eq:px-laplacian_zero} if
\begin{equation}\label{eq:weak_form_zero}
	\int\nolimits_{\Omega} \abs{\nabla u}^{p(x)-2}\nabla u \cdot \nabla \psi \, dx = 0 \quad \forall \psi \in W_0^{1,p(\cdot)}(\Omega),
\end{equation}
and subsolutions (resp. supersolutions) of \eqref{eq:px-laplacian_zero} are defined analogously. Let us define $u_\lambda(x) \coloneq (u(x) - \lambda)_+ \coloneq \max\{u(x) - \lambda, 0\}$ for $\lambda\in \mathbb{R}$.
\label{subsec:higher_integrability}
\subsection{Higher Integrability}
Before we formulate the higher integrability of the gradient, we introduce some well-known results. The following Caccioppoli inequality and Gehring-type lemma can be found in \cite[Lemma 3.1]{DieSch14} and \cite[Proposition 6]{Dieningettwein}, respectively.

\begin{lemma}\label{lem:classical_caccio}
	Let $Q \subset \Omega$ be a cube with sidelength $R$ and $1<p^-\leq p^+ < \infty$. Then the weak solution $u$ of \eqref{eq:px-laplacian_non_zero} satisfies
	\begin{align*}
		\int\nolimits_{Q} \abs{\nabla u}^{p(x)} \,dx \lesssim \bigg( \int\nolimits_{2Q} \left| \frac{u - \mean{u}_{2Q}}{R} \right|^{p(x)} \,dx + \int\nolimits_{2Q} |G|^{p(x)} \,dx \bigg)
	\end{align*}
	for all $2Q\Subset\Omega$.
\end{lemma}

\begin{theorem}[Giaquinta--Modica]\label{thm:classical_Gehring} 
Let $Q_0 \subset \mathbb{R}^n$ be a cube, $U \in L^1(Q_0)$, and $H \in L^{q_0}(Q_0)$ for some $q_0 > 1$. Suppose that for some $\theta \in (0, 1)$, $c_1 > 0$, and all cubes $Q$ with $2Q \subset Q_0$
\begin{align*}
\dashint\nolimits_Q |U| \, dx \le c_1 \bigg( \dashint\nolimits_{2Q} |U|^\theta \, dx \bigg)^{1/\theta} + \dashint\nolimits_{2Q} |H| \, dx.
\end{align*}
Then there exist $q_1 > 1$ and $c_2 > 1$ such that $U \in L_{\mathrm{loc}}^{q_1}(Q_0)$ and for all $q_2 \in[1, q_1]$
\begin{align*}
  \bigg(\dashint\nolimits_Q |U|^{q_2} \, dx \bigg)^{1/q_2} \le c_2 \dashint\nolimits_{2Q} |U| \, dx + c_2 \bigg(\dashint\nolimits_{2Q} |H|^{q_2} \, dx \bigg)^{1/q_2}.
\end{align*}
\end{theorem}
Note that both results above are true for cubes or balls. 
\begin{theorem}[Local higher integrability]
	\label{thm:local_higher_integrability}
	Let $p \in \mathcal{N} \cap \mathcal{A}$ with $1 < p^- \leq p^+ < \infty$, and let $u \in W^{1,p(\cdot)}(\Omega)$ be a weak solution of \eqref{eq:px-laplacian_non_zero}. Let $G \in L^{p(\cdot)(1+\sigma_0)}_{\loc}(\Omega;\mathbb{R}^n)$ for some $\sigma_0 > 0$. Then there exist $\delta = \delta\bigl(n, p^-, p^+, [p]_{\mathcal A}, [p]_{\mathcal N}, \sigma_0\bigr) \in (0,\sigma_0]$ and $\epsilon = \epsilon\bigl(n, p^-, p^+, [p]_{\mathcal A}, [p]_{\mathcal N}\bigr) > 0$ such that $|\nabla u| \in L^{p(\cdot)(1+\delta)}_{\loc}(\Omega)$.
	Moreover, there exists $C = C\bigl(n, p^-, p^+, [p]_{\mathcal A}, [p]_{\mathcal N}, \sigma_0\bigr) > 0$ such that for every cube $Q$ with $4Q \Subset \Omega$, $|Q|\leq1$, and $\int_{4Q} |\nabla u|^{p(x)} \, dx < 1$, we have
	\begin{align}\label{eq:local_higher}
		\left( \dashint_Q |\nabla u|^{p(x)(1+\delta)} \, dx \right)^{\frac{1}{1+\delta}} 
		\leq C \dashint_{2Q} |\nabla u|^{p(x)} \, dx 
		 + C \left( \dashint_{2Q} |G|^{p(x)(1+\delta)} \, dx \right)^{\frac{1}{1+\delta}} + C|Q|^\epsilon.
	\end{align}
\end{theorem}
\begin{proof}
	Let $s_0$ and $\epsilon$ be the same as in Theorem~\ref{thm:improved_poincare}. Consider $s\in (1,s_0]$.
	We use Theorem~\ref{thm:classical_Gehring} for $Q_0\coloneq4Q$. Let us now consider any $\tilde{Q}$ such that $2\tilde{Q}\subset Q_0$.
	By the Caccioppoli inequality (Lemma~\ref{lem:classical_caccio}),
	\begin{align*}
		\dashint_{\tilde{Q}} |\nabla u|^{p(x)}\,dx 
		&\leq C\dashint_{2\tilde{Q}} \left( \frac{|u-\mean{u}_{2\tilde{Q}}|}{\ell(\tilde{Q})} \right)^{p(x)}\,dx + C\dashint_{2\tilde{Q}}|G|^{p(x)}\,dx.
	\end{align*}
Now,
\begin{align*}
	\int_{4Q} \abs{\nabla u}^{p(x)}\,dx &< 1
\end{align*}
and $\abs{Q} \le 1$ imply
\begin{align*}
	\int_{2\widetilde{Q}} \abs{\nabla u}^{p(x)/s}\,dx &\lesssim 1.
\end{align*}
Remark~\ref{rem:smallness_of_norms} together with Theorem~\ref{thm:improved_poincare} for $s_1=s_2=s$ yields
\begin{align*}
	\dashint_{2\widetilde{Q}} \left( \frac{\abs{u-\mean{u}_{2\widetilde{Q}}}}{r_{2\widetilde{Q}}} \right)^{p(x)}\,dx 
	&\le C\left( \dashint_{2\widetilde{Q}} \abs{\nabla u}^{p(x)/s}\,dx \right)^s + C\abs{2\widetilde{Q}}^\epsilon.
\end{align*}
	Consequently,
	\begin{equation}
		\label{eq:HI_reverse_holder}
		\dashint_{\tilde{Q}}|\nabla u|^{p(x)}\,dx \le C\left( \dashint_{2\tilde{Q}}|\nabla u|^{p(x)/s}\,dx \right)^s + C\dashint_{2\tilde{Q}}|G|^{p(x)}\,dx + C|Q_0|^\epsilon.
	\end{equation}
	Using Theorem~\ref{thm:classical_Gehring} for $U\coloneq|\nabla u|^{p(x)}$, $H(x)\coloneq\bigl(|G(x)|^{p(x)}+|Q_0|^\epsilon\bigr)$ and $\theta=\frac{1}{s}$
	gives that there exists $\delta>0$ such that
	\begin{align*}
		\left( \dashint_{\tilde{Q}}|\nabla u|^{p(x)(1+\delta)}\,dx \right)^{\frac1{1+\delta}} 
		&\le C\dashint_{2\tilde{Q}}|\nabla u|^{p(x)}\,dx \\
		&\quad + C\left( \dashint_{2\tilde{Q}}|G|^{p(x)(1+\delta)}\,dx \right)^{\frac1{1+\delta}} + C|Q_0|^\epsilon.
	\end{align*}
	Setting $\tilde{Q}=Q$ proves \eqref{eq:local_higher}. Then a standard covering argument implies the desired result.
\end{proof}

\subsection{Local Boundedness}\label{subsec:Local_Bound}
 Using the tools that we developed, we obtain the following $L^\infty$ estimate. 
 \begin{theorem}[Local boundedness]\label{thm:local_bound}
Let $p \in \mathcal{A} \cap \mathcal{N}$ with $1 < p^- \leq p^+ < \infty,$ and let $u \in W^{1,p(\cdot)}(\Omega)$ be a weak solution of \eqref{eq:px-laplacian_zero}.
Let $B = B_R(x_0)$ be a ball such that $2B \Subset \Omega, \quad |B|\leq1,$ and assume that $\int_{2B} |\nabla u(y)|^{p(y)} \, dy \leq 1.$ Then there exists $c = c\bigl(n, p^-, p^+, [p]_{\mathcal{A}}, [p]_{\mathcal{N}}\bigr) > 0$ such that
\begin{align}
\label{eq:loc_bound_formula}
\frac{\|u - \mean{u}_{2B}\|_{L^\infty(B)}}{R}\leq c \left(\dashint_{2B}\left|\frac{u - \mean{u}_{2B}}{R}\right|^{p(x)}\, dx+ 1\right)^{1/p^-}.
\end{align}
\end{theorem}
Before proving Theorem~\ref{thm:local_bound}, we prove the Caccioppoli inequality, which is essential for the subsequent $L^\infty$ estimate.
\begin{lemma}[Caccioppoli inequality] \label{lem:Caccio_px}
	Let $B_r \subset B_R \subset \Omega$ be concentric balls. If $u \in W^{1,p(\cdot)}(\Omega)$ is a weak subsolution of \eqref{eq:px-laplacian_zero}, then for any $\lambda \in \RR$,
	\begin{equation*}
		\int\nolimits_{B_r} |\nabla u_\lambda|^{p(x)} \,dx \leq c \int\nolimits_{B_R} \left( \frac{u_\lambda}{R-r} \right)^{p(x)} \, dx,
	\end{equation*}
	where $c = c(n, p^-, p^+)$.
\end{lemma}
\begin{proof}
	Let $\eta \in C^\infty_c(B_R)$ be a cutoff function such that $0 \le \eta \le 1$, $\eta \equiv 1$ on $B_r$, and $|\nabla \eta| \le c/(R-r)$. Testing \eqref{eq:weak_form_zero} with $\psi = \eta^{p^+} u_\lambda$ and applying Young's inequality, we have:
	\begin{align*}
		\int\nolimits_{B_R} |\nabla u_\lambda|^{p(x)} \eta^{p^+} \,dx 
		&\le \int\nolimits_{B_R} p^+ \eta^{p^+-1} u_\lambda |\nabla u_\lambda|^{p(x)-1} |\nabla \eta| \, dx \\
		&\le \tfrac{1}{2} \int\nolimits_{B_R} |\nabla u_\lambda|^{p(x)} \eta^{p^+} \, dx + c\int\nolimits_{B_R} \left( \frac{u_\lambda}{R-r} \right)^{p(x)} \, dx.
	\end{align*}
	Absorbing the first term on the right into the left side yields the result. 
\end{proof}

\begin{remark}\label{rmk:neg_trunc}
If $u \in W^{1,p(\cdot)}(\Omega)$ is a weak supersolution of \eqref{eq:px-laplacian_zero}, then Caccioppoli holds with $u_\lambda$ replaced by $(-u)_\lambda$. 
\end{remark}
\begin{lemma}\label{lem:local_bound_future}
Let $u\in W^{1,p(\cdot)}(\Omega)$ be a weak subsolution of \eqref{eq:px-laplacian_zero}, and assume that $p\in\mathcal{A}\cap\mathcal{N}$ with $1<p^-\leq p^+<\infty$. Let $B=B_R(x_0)$ be a ball satisfying $2B\Subset\Omega, \quad \abs{B}\leq1, \quad \int_{2B}\abs{\nabla u(y)}^{p(y)}\,dy\leq 1$. Then there exist positive constants $\beta=\beta(n,p^-,p^+,[p]_{\mathcal{A}},[p]_{\mathcal{N}})$ and $c=c\bigl(n,p^-,p^+,[p]_{\mathcal{A}},[p]_{\mathcal{N}}\bigr)$ such that
	\begin{align}
		\sup_{B_R}\frac{(u-\mean{u}_{2B})_+}{R} \leq c\left[ \left(\frac{|A(0,2R)|}{R^n}\right)^\beta \dashint_{B_{2R}}\left|\frac{(u-\mean{u}_{2B})_+}{R}\right|^{p(x)}\,dx +1 \right]^{1/p^-},
	\end{align}
	where $A(0,2R) := \bigl\{ x\in 2B: u(x)-\mean{u}_{2B}>0 \bigr\}.$
\end{lemma}
 \begin{proof}
Let $R$ be the radius of $B$. Let us define $R_i := R(1+2^{-i})$, $\tilde{R}_i := (R_i+R_{i+1})/2$, and $k_i := d R(1-2^{-i})$, along with the balls $B_i := B_{R_i}$ and $\tilde{B}_i := B_{\tilde{R}_i}$ for $i = 0,1,2,\dots$, centered at $x_0$, and set $v_i := (u-\mean{u}_{2B})_{k_i},$ where $d \geq 1$ will be chosen later. Note that
	\begin{align*}
		R_i-R_{i+1} = \frac{R}{2^{i+1}}, \qquad k_{i+1}-k_i = \frac{d\,R}{2^{i+1}}.
	\end{align*}
	Let $\psi_{i+1} \in C_c^\infty({\tilde{B_i}})$ be a standard cutoff such that $\abs{\nabla \psi_{i+1}}\leq \frac{4}{R_i-R_{i+1}}$ and  $\psi_{i+1}=1$ on $B_{i+1}.$
	By the H\"older and the Sobolev--\Poincare inequalities with $A(k_i,R_j):=\{x\in B_j:\ v_i>0\}$, we obtain
	\begin{align*}
		&\dashint\nolimits_{B_{i+1}} \left|\frac{v_{i+1}}{R_{i+1}}\right|^{p(x)}\,dx \lesssim \left(\frac{|A(k_{i+1},R_{i+1})|}{R^n}\right)^{1-\frac{1}{s}}\left(\dashint\nolimits_{B_{i+1}}\left|\frac{ v_{i+1}}{R_{i+1}}\right|^{s p(x)}\,dx\right)^{\frac{1}{s}}.
  \end{align*}
Now, using Corollary \ref{cor:zero_boundary_value}, we get:
  \begin{align*}
    \lefteqn{\left(\dashint\nolimits_{B_{i+1}}\left|\frac{ v_{i+1}}{R_{i+1}}\right|^{s p(x)}\,dx\right)^{\frac{1}{s}}} \qquad &
    \\
    &\lesssim 2^{i p^+} \left(\dashint\nolimits_{\tilde{B}_{i}}\left|\frac{2^{-i}\psi_{i+1} v_{i+1}}{R_{i+1}}\right|^{s p(x)}\,dx\right)^{\frac{1}{s}} \\
		&\lesssim 2^{i p^+} \left(\dashint\nolimits_{\tilde{B}_{i}} |\nabla (2^{-i}\psi_{i+1} v_{i+1})|^{p(x)}\,dx + \frac{\bigl|\{x\in B_i:\ v_{i+1}> 0\}\bigr|}{R^n}\right)\\
		&\lesssim 2^{i p^+}\left(\dashint\nolimits_{\tilde{B}_i} \left(2^{-i} v_{i+1}|\nabla\psi_{i+1}|\right)^{p(x)} dx + \dashint\nolimits_{\tilde{B}_{i}} |2^{-i}\nabla v_{i+1}|^{p(x)} dx + \frac{|A(k_{i+1}, R_i)|}{R^n}\right) \\
		&\lesssim 2^{2ip^+} \left(\dashint\nolimits_{B_i} \left(\frac{v_{i+1}}{R_i}\right)^{p(x)} dx + \frac{|A(k_{i+1}, R_i)|}{R^n}\right)\\
		&\lesssim 2^{2ip^+} \left(\dashint\nolimits_{B_i} \left(\frac{v_i}{R_i}\right)^{p(x)} dx + \frac{|A(k_{i+1}, R_i)|}{R^n}\right).
	\end{align*}

Applying Corollary \ref{cor:zero_boundary_value} requires the bound $\int_{\widetilde B_i} |\nabla (2^{-i} \psi_{i+1} v_{i+1})|^{p(x)} \, dx \lesssim 1$ due to Remark \ref{rem:smallness_of_norms}, which we justify below using Theorem~\ref{thm:improved_poincare} (with $s_1=s_2=1$).

	 \begin{align*}
	 	\lefteqn{\int\nolimits_{\widetilde B_i} |\nabla (2^{-i} \psi_{i+1} v_{i+1})|^{p(x)} \, dx}\qquad &
	 	\\
	 	&\lesssim\int\nolimits_{\widetilde B_i} \left( 2^{-i} v_{i+1} |\nabla \psi_{i+1}| \right)^{p(x)} dx + \int_{\widetilde B_i} |\nabla v_{i+1}|^{p(x)} \, dx \\
	 	&\lesssim \int\nolimits_{2B} \left| \frac{u - \mean{u}_{2B}}{R} \right|^{p(x)} dx + \int\nolimits_{2B} |\nabla u|^{p(x)} \, dx \\
	 	&\lesssim \int\nolimits_{2B} |\nabla u|^{p(x)} \, dx + |B|^\epsilon \\
	 	&\lesssim 1.
	 \end{align*}
	 	 
On the other hand, we have that
\begin{align*}
	\frac{|A(k_{i+1}, R_i)|}{R^n} &\lesssim  \frac{1}{R^n} \int\nolimits_{A(k_{i+1}, R_i)} \left( \frac{R_i}{k_{i+1}-k_i} \right)^{p(x)} \left| \frac{v_i}{R_i} \right|^{p(x)} dx \\
	&\lesssim\,\frac{2^{i p^+}}{d^{p^-}} \dashint\nolimits_{B_i} \left| \frac{v_i}{R_i} \right|^{p(x)} dx.
\end{align*}
Substituting this bound for the measure term in our previous estimate, for $k_i < k_{i+1}$ and considering that $\frac{|A(k_{i+1}, R_{i+1})|}{R^n} \leq \frac{|A(k_{i}, R_i)|}{R^n}$, we obtain
\begin{align*}
	\lefteqn{\dashint\nolimits_{B_{i+1}}\left| \frac{v_{i+1}}{R_{i+1}} \right|^{p(x)} dx} \qquad &
	 \\
	&\leq 2^{2ip^+} \left(\frac{|A(k_i, R_i)|}{R^n}\right)^{1-\frac{1}{s}} \left(\dashint\nolimits_{B_i} \left| \frac{v_i}{R_i} \right|^{p(x)} dx +  \, \frac{2^{i p^+}}{d^{p^-}} \dashint\nolimits_{B_i} \left| \frac{v_i}{R_i} \right|^{p(x)} dx \right).
\end{align*}
Since 
\begin{align*}
 \left( \frac{|A(k_{i+1}, R_{i+1})|}{R^n} \right)^\beta\leq\left( c \, \frac{2^{i p^+}}{d^{p^-}} \dashint\nolimits_{B_i} \left| \frac{v_i}{R_i} \right|^{p(x)} dx \right)^\beta,
\end{align*}  
we obtain
\begin{align*}
	\lefteqn{\left( \frac{|A(k_{i+1}, R_{i+1})|}{R^n} \right)^\beta\dashint\nolimits_{B_{i+1}} \left| \frac{v_{i+1}}{R_{i+1}} \right|^{p(x)} dx} \qquad&
	\\
	&\leq c \,2^{3ip^+}\left( 1 + \frac{1}{d^{p^-}} \right) \left( \frac{|A(k_i, R_i)|}{R^n} \right)^{\beta(1+\beta)} \left( c \, \frac{2^{i p^+}}{d^{p^-}} \right)^\beta \left( \dashint\nolimits_{B_i} \left| \frac{v_i}{R_i} \right|^{p(x)} dx \right)^{1+\beta}.
\end{align*}
Here $\beta \in (0,1)$ is chosen such that $\beta(1+\beta) = 1-\frac{1}{s}.$

If we set 
$$W_i := \left( \frac{|A(k_i, R_i)|}{R^n} \right)^\beta \dashint_{B_i} \left| \frac{v_i}{R_i} \right|^{p(x)} \, dx,$$
then we obtain 
$$W_{i+1} \leq c \frac{2^{5p^+ i}}{d^{p^- \beta}} \left( 1 + \frac{1}{d^{p^-}} \right) W_i^{1+\beta}.$$
Since $d \geq 1$, this simplifies to
$$W_{i+1} \leq 2c d^{-p^- \beta} 2^{5p^+ i} W_i^{1+\beta}.$$
Due to the algebraic lemma \cite[Lemma 7.1]{GIUS}, $W_i \to 0$ as $i \to \infty$ provided
\begin{align}\label{eq:smallness_zero}
	W_0 \leq \left(2c d^{-p^- \beta}\right)^{-1/\beta} \left(2^{5p^+}\right)^{-1/\beta^2}.
\end{align}
Defining $K_0 := (2c)^{-1/\beta} 2^{-5p^+/\beta^2}$ and choosing
$$d^{p^-} = 1 + K_0^{-1} W_0$$
ensures that \eqref{eq:smallness_zero} holds. This yields $(u-\mean{u}_{2B}-dR)_+=0$ in $B_R$, which in turn implies
$$ \sup_{B_R} \, (u - \mean{u}_{2B})_{+} \leq d R \leq C R\left[ \left( \frac{|A(0, 2R)|}{R^n} \right)^\beta \dashint_{B_{2R}} \left| \frac{v_0}{2R} \right|^{p(x)} \, dx + 1 \right]^{1/p^-}. $$
\end{proof}

\begin{proof}[Proof of Theorem~\ref{thm:local_bound}]
	In view of Remark \ref{rmk:neg_trunc}, the same proof as in Lemma \ref{lem:local_bound_future} with $u$ replaced by $-u$ (where the set $A(0,2R)$ is instead defined as $\{x\in 2B:\ \mean{u}_{2B}-u(x)>0\}$) yields a similar upper bound for $(\mean{u}_{2B}-u)_{+}$, completing the desired $L^\infty$ estimate \eqref{eq:loc_bound_formula}, which proves Theorem~\ref{thm:local_bound}.
\end{proof}

We conclude with the following open problem.
\begin{question*}
	Let $\varphi$ be a generalized Young function \cite{DHHR} with uniform Simonenko indices $1 < p^- \leq p^+ < \infty$, i.e.
  \begin{align*}
    p^- \leq \frac{\phi'(x,t)\,t}{\phi(x,t)} \leq p^+ \qquad \text{for all $x \in \RRn$ and $t \geq 0$.}
  \end{align*}
  Suppose the Hardy--Littlewood maximal operator~$M$ is norm bounded on $L^\varphi(\RRn)$, i.e. $\norm{Mf}_\phi \lesssim \norm{f}_\phi$. Consider a (local) minimizer $u$ of the energy
	\begin{align*}
		\mathcal{J}(v) &:= \int_\Omega \varphi(x,\abs{\nabla v})\,dx.
	\end{align*}
	Is the boundedness of~$M$ sufficient to guarantee the local boundedness and/or continuity of~$u$?
\end{question*}

\printbibliography

@article {Alkhutov_Harnack_Holder_px,
    AUTHOR = {Alkhutov, Yu.\ A.},
     TITLE = {The {H}arnack inequality and the {H}\"older property of
              solutions of nonlinear elliptic equations with a nonstandard
              growth condition},
   JOURNAL = {Differ. Uravn.},
  FJOURNAL = {Differentsial\cprime nye Uravneniya},
    VOLUME = {33},
      YEAR = {1997},
    NUMBER = {12},
     PAGES = {1651--1660, 1726},
      ISSN = {0374-0641},
   MRCLASS = {35J60 (35B65 35J70)},
  MRNUMBER = {1669915},
MRREVIEWER = {Ya\ Zhe\ Chen},
}

@article {Adamowicz_Toivanen_Hoelder_nonstandard,
    AUTHOR = {Adamowicz, Tomasz and Toivanen, Olli},
     TITLE = {H\"older continuity of quasiminimizers with nonstandard
              growth},
   JOURNAL = {Nonlinear Anal.},
  FJOURNAL = {Nonlinear Analysis. Theory, Methods \& Applications. An
              International Multidisciplinary Journal},
    VOLUME = {125},
      YEAR = {2015},
     PAGES = {433--456},
      ISSN = {0362-546X,1873-5215},
   MRCLASS = {49N60 (35B65 35J20 35J62)},
  MRNUMBER = {3373594},
MRREVIEWER = {Eugen\ Viszus},
       DOI = {10.1016/j.na.2015.05.023},
       URL = {https://doi.org/10.1016/j.na.2015.05.023},
}

@article {Lerner05,
    AUTHOR = {Lerner, Andrei K.},
     TITLE = {Some remarks on the {H}ardy-{L}ittlewood maximal function on
              variable {$L^p$} spaces},
   JOURNAL = {Math. Z.},
  FJOURNAL = {Mathematische Zeitschrift},
    VOLUME = {251},
      YEAR = {2005},
    NUMBER = {3},
     PAGES = {509--521},
      ISSN = {0025-5874,1432-1823},
   MRCLASS = {42B25},
  MRNUMBER = {2190341},
MRREVIEWER = {Charles\ N.\ Moore},
       DOI = {10.1007/s00209-005-0818-5},
       URL = {https://doi.org/10.1007/s00209-005-0818-5},
}

@article {Piatcoscia_Hoelder_Continuity,
    AUTHOR = {Chiad\`o{} Piat, Valeria and Coscia, Alessandra},
     TITLE = {H\"older continuity of minimizers of functionals with variable
              growth exponent},
   JOURNAL = {Manuscripta Math.},
  FJOURNAL = {Manuscripta Mathematica},
    VOLUME = {93},
      YEAR = {1997},
    NUMBER = {3},
     PAGES = {283--299},
      ISSN = {0025-2611,1432-1785},
   MRCLASS = {49N60 (49J10)},
  MRNUMBER = {1457729},
MRREVIEWER = {Elvira\ Mascolo},
       DOI = {10.1007/BF02677472},
       URL = {https://doi.org/10.1007/BF02677472},
}

@article {Nekvinda2004,
    AUTHOR = {Nekvinda, Ale\v s},
     TITLE = {Hardy-{L}ittlewood maximal operator on {$L^{p(x)}(\mathbb{R})$}},
   JOURNAL = {Math. Inequal. Appl.},
  FJOURNAL = {Mathematical Inequalities \& Applications},
    VOLUME = {7},
      YEAR = {2004},
    NUMBER = {2},
     PAGES = {255--265},
      ISSN = {1331-4343,1848-9966},
   MRCLASS = {42B25 (26D15 46E30 47B38)},
  MRNUMBER = {2057644},
MRREVIEWER = {Jos\'e\ Mar\'ia\ Martell},
       DOI = {10.7153/mia-07-28},
       URL = {https://doi.org/10.7153/mia-07-28},
}

@book {CUF,
    AUTHOR = {Cruz-Uribe, David V. and Fiorenza, Alberto},
     TITLE = {Variable {L}ebesgue spaces},
    SERIES = {Applied and Numerical Harmonic Analysis},
      NOTE = {Foundations and harmonic analysis},
 PUBLISHER = {Birk\-h\"{a}user/Springer, Heidelberg},
      YEAR = {2013},
     PAGES = {x+312},
}

@book {Grafakos2014modern,
    AUTHOR = {Grafakos, Loukas},
     TITLE = {Modern {F}ourier analysis},
    SERIES = {Graduate Texts in Mathematics},
    VOLUME = {250},
   EDITION = {Third},
 PUBLISHER = {Springer, New York},
      YEAR = {2014},
     PAGES = {xvi+624},
      ISBN = {978-1-4939-1229-2},
   MRCLASS = {42-01 (42Bxx)},
  MRNUMBER = {3243741},
MRREVIEWER = {Atanas\ G.\ Stefanov},
       DOI = {10.1007/978-1-4939-1230-8},
       URL = {https://doi.org/10.1007/978-1-4939-1230-8},
}

@article {BDS,
    AUTHOR = {Balci, Anna Kh. and Diening, Lars and Surnachev, Mikhail},
     TITLE = {New examples on {L}avrentiev gap using fractals},
   JOURNAL = {Calc. Var. Partial Differential Equations},
  FJOURNAL = {Calculus of Variations and Partial Differential Equations},
    VOLUME = {59},
      YEAR = {2020},
    NUMBER = {5},
     PAGES = {Paper No. 180, 34},
      ISSN = {0944-2669,1432-0835},
   MRCLASS = {35J20 (35J62 46E35)},
  MRNUMBER = {4153906},
MRREVIEWER = {Steven\ George\ Krantz},
       DOI = {10.1007/s00526-020-01818-1},
       URL = {https://doi.org/10.1007/s00526-020-01818-1},
}

@article {DieSch14,
    AUTHOR = {Diening, L. and Schwarzacher, S.},
     TITLE = {Global gradient estimates for the {$p(\cdot)$}-{L}aplacian},
   JOURNAL = {Nonlinear Anal.},
  FJOURNAL = {Nonlinear Analysis. Theory, Methods \& Applications. An
              International Multidisciplinary Journal},
    VOLUME = {106},
      YEAR = {2014},
     PAGES = {70--85},
      ISSN = {0362-546X,1873-5215},
   MRCLASS = {35J62 (35B45 35B65)},
  MRNUMBER = {3209686},
       DOI = {10.1016/j.na.2014.04.006},
       URL = {https://doi.org/10.1016/j.na.2014.04.006},
}

@article {BreDieSch15,
    AUTHOR = {Breit, Dominic and Diening, Lars and Schwarzacher, Sebastian},
     TITLE = {Finite element approximation of the {$p(\cdot)$}-{L}aplacian},
   JOURNAL = {SIAM J. Numer. Anal.},
  FJOURNAL = {SIAM Journal on Numerical Analysis},
    VOLUME = {53},
      YEAR = {2015},
    NUMBER = {1},
     PAGES = {551--572},
      ISSN = {0036-1429,1095-7170},
   MRCLASS = {65N30 (35J60 46E30 65N15)},
  MRNUMBER = {3313830},
MRREVIEWER = {Marius\ Ghergu},
       DOI = {10.1137/130946046},
       URL = {https://doi.org/10.1137/130946046},
}

@book {DHHR,
    AUTHOR = {Diening, Lars and Harjulehto, Petteri and H\"{a}st\"{o}, Peter
              and R{\r u}\v{z}i\v{c}ka, Michael},
     TITLE = {Lebesgue and {S}obolev spaces with variable exponents},
    SERIES = {Lecture Notes in Mathematics},
    VOLUME = {2017},
 PUBLISHER = {Springer, Heidelberg},
      YEAR = {2011},
     PAGES = {x+509},
      ISBN = {978-3-642-18362-1},
   MRCLASS = {46-02 (26D10 31B15 35J60 35Q35 46E30 46E35 46N20)},
  MRNUMBER = {2790542},
MRREVIEWER = {Dorothee\ D.\ Haroske},
       DOI = {10.1007/978-3-642-18363-8},
       URL = {https://doi.org/10.1007/978-3-642-18363-8},
}

@article {HastoOk2022,
    AUTHOR = {H\"ast\"o, Peter and Ok, Jihoon},
     TITLE = {Maximal regularity for local minimizers of non-autonomous
              functionals},
   JOURNAL = {J. Eur. Math. Soc. (JEMS)},
  FJOURNAL = {Journal of the European Mathematical Society (JEMS)},
    VOLUME = {24},
      YEAR = {2022},
    NUMBER = {4},
     PAGES = {1285--1334},
      ISSN = {1435-9855,1435-9863},
   MRCLASS = {49N60 (35A15 35J62 46E30 49J10)},
  MRNUMBER = {4397041},
MRREVIEWER = {Antonia\ Passarelli di Napoli},
       DOI = {10.4171/JEMS/1118},
       URL = {https://doi.org/10.4171/JEMS/1118},
}

@article {Spanne65,
    AUTHOR = {Spanne, Sven},
     TITLE = {Some function spaces defined using the mean oscillation over
              cubes},
   JOURNAL = {Ann. Scuola Norm. Sup. Pisa Cl. Sci. (3)},
  FJOURNAL = {Annali della Scuola Normale Superiore di Pisa. Classe di
              Scienze. Serie III},
    VOLUME = {19},
      YEAR = {1965},
     PAGES = {593--608},
      ISSN = {0391-173X},
   MRCLASS = {46.38},
  MRNUMBER = {190729},
MRREVIEWER = {L.\ C.\ Young},
}

@article {AcerbiMingione2005,
    AUTHOR = {Acerbi, Emilio and Mingione, Giuseppe},
     TITLE = {Gradient estimates for the {$p(x)$}-{L}aplacean system},
   JOURNAL = {J. Reine Angew. Math.},
  FJOURNAL = {Journal f\"{u}r die Reine und Angewandte Mathematik. [Crelle's
              Journal]},
    VOLUME = {584},
      YEAR = {2005},
     PAGES = {117--148},
      ISSN = {0075-4102,1435-5345},
   MRCLASS = {35J60 (35B45 35B65)},
  MRNUMBER = {2155087},
MRREVIEWER = {Eugen\ Viszus},
       DOI = {10.1515/crll.2005.2005.584.117},
       URL = {https://doi.org/10.1515/crll.2005.2005.584.117},
}

@article {Kopaliani,
    AUTHOR = {Kopaliani, Tengizi S.},
     TITLE = {Infimal convolution and {M}uckenhoupt {$A_{p(\cdot)}$}
              condition in variable {$L^p$} spaces},
   JOURNAL = {Arch. Math. (Basel)},
  FJOURNAL = {Archiv der Mathematik},
    VOLUME = {89},
      YEAR = {2007},
    NUMBER = {2},
     PAGES = {185--192},
      ISSN = {0003-889X,1420-8938},
   MRCLASS = {42B20 (42B25 46E30)},
  MRNUMBER = {2341730},
MRREVIEWER = {Paul\ Alton\ Hagelstein},
       DOI = {10.1007/s00013-007-2035-4},
       URL = {https://doi.org/10.1007/s00013-007-2035-4},
}

@article {Lerner2010questions,
    AUTHOR = {Lerner, Andrei K.},
     TITLE = {On some questions related to the maximal operator on variable
              {$L^p$} spaces},
   JOURNAL = {Trans. Amer. Math. Soc.},
  FJOURNAL = {Transactions of the American Mathematical Society},
    VOLUME = {362},
      YEAR = {2010},
    NUMBER = {8},
     PAGES = {4229--4242},
      ISSN = {0002-9947,1088-6850},
   MRCLASS = {42B25 (46E30)},
  MRNUMBER = {2608404},
MRREVIEWER = {Maria\ J.\ Carro},
       DOI = {10.1090/S0002-9947-10-05066-X},
       URL = {https://doi.org/10.1090/S0002-9947-10-05066-X},
}

@book {Maly,
    AUTHOR = {Mal\'y, Jan and Ziemer, William P.},
     TITLE = {Fine regularity of solutions of elliptic partial differential
              equations},
    SERIES = {Mathematical Surveys and Monographs},
    VOLUME = {51},
 PUBLISHER = {American Mathematical Society, Providence, RI},
      YEAR = {1997},
     PAGES = {xiv+291},
      ISBN = {0-8218-0335-2},
   MRCLASS = {35J65 (31C45 35B65 35J70 49N60)},
  MRNUMBER = {1461542},
MRREVIEWER = {Tero\ Kilpel\"ainen},
       DOI = {10.1090/surv/051},
       URL = {https://doi.org/10.1090/surv/051},
}

@misc{adamadzedieningkopalianiok,
      title={Double phase meets Muckenhoupt}, 
      author={Daviti Adamadze and Lars Diening and Tengiz Kopaliani and Jihoon Ok},
      year={2026},
      eprint={2601.20736},
      archivePrefix={arXiv},
      url={https://arxiv.org/abs/2601.20736}, 
}

@book {GIUS,
    AUTHOR = {Giusti, Enrico},
     TITLE = {Direct methods in the calculus of variations},
 PUBLISHER = {World Scientific Publishing Co., Inc., River Edge, NJ},
      YEAR = {2003},
     PAGES = {viii+403},
      ISBN = {981-238-043-4},
   MRCLASS = {49-02 (35J50 49J10 49K10 49N60)},
  MRNUMBER = {1962933},
MRREVIEWER = {Giovanni\ Alberti},
       DOI = {10.1142/9789812795557},
       URL = {https://doi.org/10.1142/9789812795557},
}

@article {AdamadzeDieningKopaliani2026,
    AUTHOR = {Adamadze, Daviti and Diening, Lars and Kopaliani, Tengiz},
     TITLE = {Maximal operator on variable exponent spaces},
   JOURNAL = {J. Math. Anal. Appl.},
  FJOURNAL = {Journal of Mathematical Analysis and Applications},
    VOLUME = {554},
      YEAR = {2026},
    NUMBER = {2},
     PAGES = {Paper No. 129972, 21},
      ISSN = {0022-247X,1096-0813},
   MRCLASS = {42B25 (42B35 46E30)},
  MRNUMBER = {4951202},
       DOI = {10.1016/j.jmaa.2025.129972},
       URL = {https://doi.org/10.1016/j.jmaa.2025.129972},
}

@article {AcMi1,
    AUTHOR = {Acerbi, Emilio and Mingione, Giuseppe},
     TITLE = {Regularity results for a class of functionals with
              non-standard growth},
   JOURNAL = {Arch. Ration. Mech. Anal.},
  FJOURNAL = {Archive for Rational Mechanics and Analysis},
    VOLUME = {156},
      YEAR = {2001},
    NUMBER = {2},
     PAGES = {121--140},
      ISSN = {0003-9527,1432-0673},
   MRCLASS = {49N60 (49J10)},
  MRNUMBER = {1814973},
MRREVIEWER = {Michele\ Carriero},
       DOI = {10.1007/s002050100117},
       URL = {https://doi.org/10.1007/s002050100117},
}

@article {Dieningettwein,
    AUTHOR = {Diening, Lars and Ettwein, Frank},
     TITLE = {Fractional estimates for non-differentiable elliptic systems
              with general growth},
   JOURNAL = {Forum Math.},
  FJOURNAL = {Forum Mathematicum},
    VOLUME = {20},
      YEAR = {2008},
    NUMBER = {3},
     PAGES = {523--556},
      ISSN = {0933-7741,1435-5337},
   MRCLASS = {35J55 (35D10 35J60)},
  MRNUMBER = {2418205},
MRREVIEWER = {Eugen\ Viszus},
       DOI = {10.1515/FORUM.2008.027},
       URL = {https://doi.org/10.1515/FORUM.2008.027},
}

@article {Zhik5,
    AUTHOR = {Zhikov, Vasili\u i\ V.},
     TITLE = {On {L}avrentiev's phenomenon},
   JOURNAL = {Russian J. Math. Phys.},
  FJOURNAL = {Russian Journal of Mathematical Physics},
    VOLUME = {3},
      YEAR = {1995},
    NUMBER = {2},
     PAGES = {249--269},
      ISSN = {1061-9208},
   MRCLASS = {49J45 (49J10)},
  MRNUMBER = {1350506},
MRREVIEWER = {Philip\ D.\ Loewen},
}

@article {Fanzhao_De_Giorgi_classes,
    AUTHOR = {Fan, Xianling and Zhao, Dun},
     TITLE = {A class of {D}e {G}iorgi type and {H}\"older continuity},
   JOURNAL = {Nonlinear Anal.},
  FJOURNAL = {Nonlinear Analysis. Theory, Methods \& Applications. An
              International Multidisciplinary Journal},
    VOLUME = {36},
      YEAR = {1999},
    NUMBER = {3},
     PAGES = {295--318},
      ISSN = {0362-546X,1873-5215},
   MRCLASS = {49N60 (49J10)},
  MRNUMBER = {1688232},
MRREVIEWER = {Martin\ Fuchs},
       DOI = {10.1016/S0362-546X(97)00628-7},
       URL = {https://doi.org/10.1016/S0362-546X(97)00628-7},
}

@article {ELM,
    AUTHOR = {Esposito, Luca and Leonetti, Francesco and Mingione, Giuseppe},
     TITLE = {Sharp regularity for functionals with {$(p,q)$} growth},
   JOURNAL = {J. Differential Equations},
  FJOURNAL = {Journal of Differential Equations},
    VOLUME = {204},
      YEAR = {2004},
    NUMBER = {1},
     PAGES = {5--55},
      ISSN = {0022-0396,1090-2732},
   MRCLASS = {49J10 (49N60)},
  MRNUMBER = {2076158},
MRREVIEWER = {Delfim\ F. M. Torres},
       DOI = {10.1016/j.jde.2003.11.007},
       URL = {https://doi.org/10.1016/j.jde.2003.11.007},
}

@article {Diening2005,
    AUTHOR = {Diening, Lars},
     TITLE = {Maximal function on {M}usielak-{O}rlicz spaces and generalized
              {L}ebesgue spaces},
   JOURNAL = {Bull. Sci. Math.},
  FJOURNAL = {Bulletin des Sciences Math\'ematiques},
    VOLUME = {129},
      YEAR = {2005},
    NUMBER = {8},
     PAGES = {657--700},
      ISSN = {0007-4497,1952-4773},
   MRCLASS = {46E30 (42B20 42B25 47B38)},
  MRNUMBER = {2166733},
MRREVIEWER = {Ale\v s\ Nekvinda},
       DOI = {10.1016/j.bulsci.2003.10.003},
       URL = {https://doi.org/10.1016/j.bulsci.2003.10.003},
}

@article {Fanzhao_discont_exponent,
    AUTHOR = {Fan, Xianling and Zhao, Dun},
     TITLE = {Regularity of quasi-minimizers of integral functionals with
              discontinuous {$p(x)$}-growth conditions},
   JOURNAL = {Nonlinear Anal.},
  FJOURNAL = {Nonlinear Analysis. Theory, Methods \& Applications. An
              International Multidisciplinary Journal},
    VOLUME = {65},
      YEAR = {2006},
    NUMBER = {8},
     PAGES = {1521--1531},
      ISSN = {0362-546X,1873-5215},
   MRCLASS = {49N60 (49J10)},
  MRNUMBER = {2246352},
MRREVIEWER = {Antonia\ Passarelli di Napoli},
       DOI = {10.1016/j.na.2005.10.027},
       URL = {https://doi.org/10.1016/j.na.2005.10.027},
}

@article {Mar,
    AUTHOR = {Marcellini, Paolo},
     TITLE = {Regularity and existence of solutions of elliptic equations
              with {$p,q$}-growth conditions},
   JOURNAL = {J. Differential Equations},
  FJOURNAL = {Journal of Differential Equations},
    VOLUME = {90},
      YEAR = {1991},
    NUMBER = {1},
     PAGES = {1--30},
      ISSN = {0022-0396,1090-2732},
   MRCLASS = {35J15 (35D10)},
  MRNUMBER = {1094446},
MRREVIEWER = {Philip\ W.\ Schaefer},
       DOI = {10.1016/0022-0396(91)90158-6},
       URL = {https://doi.org/10.1016/0022-0396(91)90158-6},
}

@book {HastoHarjulehto2019,
    AUTHOR = {Harjulehto, Petteri and H\"ast\"o, Peter},
     TITLE = {Orlicz spaces and generalized {O}rlicz spaces},
    SERIES = {Lecture Notes in Mathematics},
    VOLUME = {2236},
 PUBLISHER = {Springer, Cham},
      YEAR = {2019},
     PAGES = {x+167},
      ISBN = {978-3-030-15099-0},
   MRCLASS = {46-02 (42B35 46E35)},
  MRNUMBER = {3931352},
MRREVIEWER = {Karol\ Le\'snik},
       DOI = {10.1007/978-3-030-15100-3},
       URL = {https://doi.org/10.1007/978-3-030-15100-3},
}

@incollection {BalciDieningSurnachev2025,
    AUTHOR = {Balci, Anna and Diening, Lars and Surnachev, Mikhail},
     TITLE = {Scalar minimizers with maximal singular sets and lack of
              {M}eyers property},
 BOOKTITLE = {Friends in partial differential equations---the {N}ina {N}.
              {U}raltseva 90th anniversary volume},
     PAGES = {1--43},
 PUBLISHER = {EMS Press, Berlin},
      YEAR = {2025},
      ISBN = {978-3-98547-094-5},
   MRCLASS = {49J10 (35B65 35G20 46E35 49N60)},
  MRNUMBER = {4967630},
}

@article {John_Nirenberg_BMO,
    AUTHOR = {John, F. and Nirenberg, L.},
     TITLE = {On functions of bounded mean oscillation},
   JOURNAL = {Comm. Pure Appl. Math.},
  FJOURNAL = {Communications on Pure and Applied Mathematics},
    VOLUME = {14},
      YEAR = {1961},
     PAGES = {415--426},
      ISSN = {0010-3640,1097-0312},
   MRCLASS = {26.00},
  MRNUMBER = {131498},
MRREVIEWER = {L.\ C.\ Young},
       DOI = {10.1002/cpa.3160140317},
       URL = {https://doi.org/10.1002/cpa.3160140317},
}

@article {hirsch_schaeffner_Growth_conditions_and_regularity,
    AUTHOR = {Hirsch, Jonas and Sch\"affner, Mathias},
     TITLE = {Growth conditions and regularity, an optimal local boundedness
              result},
   JOURNAL = {Commun. Contemp. Math.},
  FJOURNAL = {Communications in Contemporary Mathematics},
    VOLUME = {23},
      YEAR = {2021},
    NUMBER = {3},
     PAGES = {Paper No. 2050029, 17},
      ISSN = {0219-1997,1793-6683},
   MRCLASS = {49N60 (35J60 49J10)},
  MRNUMBER = {4216424},
MRREVIEWER = {Teresa\ Isernia},
       DOI = {10.1142/S0219199720500297},
       URL = {https://doi.org/10.1142/S0219199720500297},
}

@article {DieningSchwarzacher2013keyestimate,
    AUTHOR = {Diening, L. and Schwarzacher, S.},
     TITLE = {On the key estimate for variable exponent spaces},
   JOURNAL = {Azerb. J. Math.},
  FJOURNAL = {Azerbaijan Journal of Mathematics},
    VOLUME = 3,
      YEAR = 2013,
    NUMBER = 2,
     PAGES = {62--69},
      ISSN = {2218-6816,2221-9501},
   MRCLASS = {42B35},
  MRNUMBER = 3084933,
MRREVIEWER = {Patrik\ Wahlberg},
}

@article {Die2,
    AUTHOR = {Diening, L.},
     TITLE = {Maximal function on generalized {L}ebesgue spaces
              {$L^{p(\cdot)}$}},
   JOURNAL = {Math. Inequal. Appl.},
  FJOURNAL = {Mathematical Inequalities \& Applications},
    VOLUME = {7},
      YEAR = {2004},
    NUMBER = {2},
     PAGES = {245--253},
      ISSN = {1331-4343,1848-9966},
   MRCLASS = {42B25 (46E30)},
  MRNUMBER = {2057643},
MRREVIEWER = {Lubo\v s\ Pick},
       DOI = {10.7153/mia-07-27},
       URL = {https://doi.org/10.7153/mia-07-27},
}

@article {FabesKenigSerapioni1982,
    AUTHOR = {Fabes, Eugene B. and Kenig, Carlos E. and Serapioni, Raul P.},
     TITLE = {The local regularity of solutions of degenerate elliptic
              equations},
   JOURNAL = {Comm. Partial Differential Equations},
  FJOURNAL = {Communications in Partial Differential Equations},
    VOLUME = {7},
      YEAR = {1982},
    NUMBER = {1},
     PAGES = {77--116},
      ISSN = {0360-5302,1532-4133},
   MRCLASS = {35J70 (35D10)},
  MRNUMBER = {643158},
MRREVIEWER = {M.-T. Lacroix},
       DOI = {10.1080/03605308208820218},
       URL = {https://doi.org/10.1080/03605308208820218},
}

@article {CruzUribeMoenNaibo2013,
    AUTHOR = {Cruz-Uribe, David and Moen, Kabe and Naibo, Virginia},
     TITLE = {Regularity of solutions to degenerate {$p$}-{L}aplacian
              equations},
   JOURNAL = {J. Math. Anal. Appl.},
  FJOURNAL = {Journal of Mathematical Analysis and Applications},
    VOLUME = {401},
      YEAR = {2013},
    NUMBER = {1},
     PAGES = {458--478},
      ISSN = {0022-247X,1096-0813},
   MRCLASS = {35J70 (35B65 35D30 35H10 35J62 35R03)},
  MRNUMBER = {3011287},
MRREVIEWER = {Francesco\ Della Pietra},
       DOI = {10.1016/j.jmaa.2012.12.023},
       URL = {https://doi.org/10.1016/j.jmaa.2012.12.023},
}

\end{document}